\documentclass[11pt]{amsart}
\usepackage{amssymb}
\usepackage{graphics}
\usepackage{latexsym}
\usepackage{amsmath}
\usepackage{amssymb,amsthm,amsfonts}
\usepackage{diagrams}
\usepackage{amscd}
\usepackage[arrow, matrix, curve]{xy}
\usepackage{syntonly}
\title[Zariski density of the monodromy groups]{The monodromy groups of Dolgachev's CY moduli spaces are Zariski dense}
\author[Mao Sheng]{Mao Sheng}
\author[Jinxing Xu]{Jinxing Xu}
\email{msheng@ustc.edu.cn}\email{xujx02@ustc.edu.cn}
\address{School of Mathematical Sciences,
University of Science and Technology of China, Hefei, 230026, China}
\author[Kang Zuo]{Kang Zuo}
\email{zuok@uni-mainz.de}
\address{Institut f\"{u}r  Mathematik, Universit\"{a}t
Mainz, Mainz, 55099, Germany}

\begin{document}
\theoremstyle{plain}
\newtheorem{thm}{Theorem}[section]
\newtheorem{theorem}[thm]{Theorem}
\newtheorem{lemma}[thm]{Lemma}
\newtheorem{corollary}[thm]{Corollary}
\newtheorem{proposition}[thm]{Proposition}
\newtheorem{addendum}[thm]{Addendum}
\newtheorem{variant}[thm]{Variant}
\theoremstyle{definition}
\newtheorem{construction}[thm]{Construction}
\newtheorem{notations}[thm]{Notations}
\newtheorem{question}[thm]{Question}
\newtheorem{problem}[thm]{Problem}
\newtheorem{remark}[thm]{Remark}
\newtheorem{remarks}[thm]{Remarks}
\newtheorem{definition}[thm]{Definition}
\newtheorem{claim}[thm]{Claim}
\newtheorem{assumption}[thm]{Assumption}
\newtheorem{assumptions}[thm]{Assumptions}
\newtheorem{properties}[thm]{Properties}
\newtheorem{example}[thm]{Example}
\newtheorem{conjecture}[thm]{Conjecture}
\numberwithin{equation}{thm}

\newcommand{\sA}{{\mathcal A}}
\newcommand{\sB}{{\mathcal B}}
\newcommand{\sC}{{\mathcal C}}
\newcommand{\sD}{{\mathcal D}}
\newcommand{\sE}{{\mathcal E}}
\newcommand{\sF}{{\mathcal F}}
\newcommand{\sG}{{\mathcal G}}
\newcommand{\sH}{{\mathcal H}}
\newcommand{\sI}{{\mathcal I}}
\newcommand{\sJ}{{\mathcal J}}
\newcommand{\sK}{{\mathcal K}}
\newcommand{\sL}{{\mathcal L}}
\newcommand{\sM}{{\mathcal M}}
\newcommand{\sN}{{\mathcal N}}
\newcommand{\sO}{{\mathcal O}}
\newcommand{\sP}{{\mathcal P}}
\newcommand{\sQ}{{\mathcal Q}}
\newcommand{\sR}{{\mathcal R}}
\newcommand{\sS}{{\mathcal S}}
\newcommand{\sT}{{\mathcal T}}
\newcommand{\sU}{{\mathcal U}}
\newcommand{\sV}{{\mathcal V}}
\newcommand{\sW}{{\mathcal W}}
\newcommand{\sX}{{\mathcal X}}
\newcommand{\sY}{{\mathcal Y}}
\newcommand{\sZ}{{\mathcal Z}}
\newcommand{\A}{{\mathbb A}}
\newcommand{\B}{{\mathbb B}}
\newcommand{\C}{{\mathbb C}}
\newcommand{\D}{{\mathbb D}}
\newcommand{\E}{{\mathbb E}}
\newcommand{\F}{{\mathbb F}}
\newcommand{\G}{{\mathbb G}}
\newcommand{\HH}{{\mathbb H}}
\newcommand{\I}{{\mathbb I}}
\newcommand{\J}{{\mathbb J}}
\renewcommand{\L}{{\mathbb L}}
\newcommand{\M}{{\mathbb M}}
\newcommand{\N}{{\mathbb N}}
\renewcommand{\P}{{\mathbb P}}
\newcommand{\Q}{{\mathbb Q}}
\newcommand{\R}{{\mathbb R}}
\newcommand{\SSS}{{\mathbb S}}
\newcommand{\T}{{\mathbb T}}
\newcommand{\U}{{\mathbb U}}
\newcommand{\V}{{\mathbb V}}
\newcommand{\W}{{\mathbb W}}
\newcommand{\X}{{\mathbb X}}
\newcommand{\Y}{{\mathbb Y}}
\newcommand{\Z}{{\mathbb Z}}
\newcommand{\id}{{\rm id}}
\newcommand{\rank}{{\rm rank}}
\newcommand{\END}{{\mathbb E}{\rm nd}}
\newcommand{\End}{{\rm End}}
\newcommand{\Hom}{{\rm Hom}}
\newcommand{\Hg}{{\rm Hg}}
\newcommand{\tr}{{\rm tr}}
\newcommand{\Sl}{{\rm Sl}}
\newcommand{\Gl}{{\rm Gl}}
\newcommand{\Cor}{{\rm Cor}}
\newcommand{\Aut}{\mathrm{Aut}}
\newcommand{\Sym}{\mathrm{Sym}}
\newcommand{\ModuliCY}{\mathfrak{M}_{CY}}
\newcommand{\HyperCY}{\mathfrak{H}_{CY}}
\newcommand{\ModuliAR}{\mathfrak{M}_{AR}}
\newcommand{\Modulione}{\mathfrak{M}_{1,n+3}}
\newcommand{\Modulin}{\mathfrak{M}_{n,n+3}}
\newcommand{\Gal}{\mathrm{Gal}}
\newcommand{\Spec}{\mathrm{Spec}}
\newcommand{\Jac}{\mathrm{Jac}}

\newcommand{\proofend}{\hspace*{13cm} $\square$ \\}
\maketitle

\begin{abstract}
Let $\mathcal{M}_{n,2n+2}$ be the coarse moduli space of CY manifolds arising from a crepant resolution of double covers of $\P^n$ branched along $2n+2$ hyperplanes in general position. We show that the monodromy group of a good family for $\mathcal{M}_{n,2n+2}$ is Zariski dense in the corresponding symplectic or orthogonal group if $n\geq 3$. In particular, the period map does not give a uniformization of any partial compactification of the coarse moduli space as a Shimura variety whenever $n\geq 3$. This disproves a conjecture of Dolgachev. As a consequence, the fundamental group of the coarse moduli space of $m$ ordered points in $\P^n$ is shown to be large once it is not a point. Similar Zariski-density result is obtained for moduli spaces of CY manifolds arising from cyclic covers of $\P^n$ branched along $m$ hyperplanes in general position. A classification towards the geometric realization problem of B. Gross for type $A$ bounded symmetric domains is given.
\end{abstract} 

\section{Introduction}\label{section:introduction}

Among all moduli spaces of algebraic varieties, the moduli spaces of
hyperplane arrangements in a projective space make a classical object of study (see \cite{DO} for a nice account from the point of view of GIT). The simplest nontrivial example of this class is the moduli space of four points in $\P^1$,
that is well-known to be identified with the moduli space of
elliptic curves with level two structure. An elliptic curve is
obtained by taking the double cover of $\P^1$ branched at four
distinct points and, via this construction, the variation of the four
points in $\P^1$ is reflected into the variation of the Hodge structures
attached to elliptic curves. The notion of variation of Hodge structure (or equivalently period map from the analytic point of view) as introduced by P. Griffiths in general has proven to be quite effective in several important geometric questions on moduli spaces of algebraic varieties. It is well known that the period map associating the
normalized period of the corresponding elliptic curve to the cross
ratio of a four pointed set in $\P^1$ is a modular form. This
fascinating idea is very successful in showing some (partial
compactifications of) moduli spaces are Shimura varieties (see e.g. \cite{ACT}, \cite{DKV}, \cite{DK}, \cite{K-R},  \cite{La}). The current paper concerns the following
\begin{conjecture}[I. Dolgachev \cite{Bo}]\label{Dolgachev conjecture}
The period space of a family of CY $n$-folds which is obtained by a resolution of double covers of $\P^n$ branched along $2n+2$ hyperplanes in general position is the complement of a $\Gamma=GL(2n,\Z[i])$-automorphic form on the type $A$ tube domain $D^{I}_{n,n}$.
\end{conjecture}

This conjecture is naturally connected with the geometric realization problem posed by B. Gross \S8 \cite{G} on the canonical polarized variation of Hodge structure (abbreviated as PVHS) over a tube domain.
 Let $\mathfrak{M}_{AR}$ be the coarse moduli space of ordered $2n+2$ hyperplane arrangements  in $\P^n$ in general position and $\tilde{\mathcal{X}}_{AR}\xrightarrow{\tilde{f}} \mathfrak{M}_{AR}$ be the family of CY $n$-folds which is obtained by a resolution of double covers of $\P^n$ branched along $2n+2$ hyperplanes in general position. This family gives a weight $n$ $\Q$-PVHS $\tilde{\V}_{AR}=(R^n\tilde{f}_{*}\Q)_{pr}$ over $\mathfrak{M}_{AR}$. Let $\V_{can}$ be the canonical $\C$-PVHS over the type $A$ tube domain $D^{I}_{n,n}$. Then we disprove Conjecture \ref{Dolgachev conjecture} in $n\geq 3$ cases by showing the following
 \begin{theorem}\label{thm: introduction not factor canonically}
 If $n\geq 3$, $\tilde{\V}_{AR}$ does not factor through $\V_{can}$.
 \end{theorem}
 By this we mean that there does not exist a nonempty analytically open subset $U\subset \mathfrak{M}_{AR}$ and a holomorphic map $j: U\rightarrow D_{n,n}^I$, such that $\tilde{\V}_{AR}\otimes \C\mid_{U}\simeq j^{*}\V_{can}$ as $\C$-PVHS.

The $n=1$ case is the classical modular family of elliptic curves and the $n=2$ case is treated in \cite{MSY}, where it was shown among other things that the weight two PVHS, modulo the constant part, factors through $\V_{can}$ canonically. The $n=3$ case is the turning point of Conjecture \ref{Dolgachev conjecture} which was shown to be false in \cite{GSSZ}. The idea is to compare the characteristic subvarieties associated to $\tilde{\V}_{AR}$ and $\V_{can}$. However, the required information on the characteristic subvariety of $\tilde{\V}_{AR}$ was attained only with the aid of a computational commutative algebra program. This obvious drawback prevented us from proceeding further in the general case. One valuable computation made in the current paper is that the characteristic subvariety of $\tilde{\V}_{AR}$ can be determined by hand at the generic point of the underlying moduli space. see Section \ref{subsection:Jacobian ring} for detail.

Surprisingly, building upon Theorem \ref{thm: introduction not factor canonically}, we can conclude further that the monodromy group of $\tilde{\V}_{AR}$ is actually Zariski dense. Take an $s\in \mathfrak{M}_{AR}$ to be the base point and let
$$
\rho: \pi_1(\mathfrak{M}_{AR}, s)\rightarrow Aut(V,Q)
$$
be the monodromy representation associated to $\tilde{\V}_{AR}$, where $V=\tilde{\V}_{AR,s}$ is the fiber of $\tilde{\V}_{AR}$ over $s$, and $Q$ is the bilinear form on $V$ induced by the natural polarization. We denote the Zariski closure of the image of $\rho$ by $Mon$. Let $Mon^0$ and $Aut^0(V,Q)$ be the identity component of $Mon$ and $Aut(V,Q)$ respectively. We will prove
\begin{theorem}\label{thm:introduction Zarisky density}
If $n\geq 3$, then $Mon^0=Aut^0(V,Q)$. That is, the monodromy group of $\tilde{\V}_{AR}$ is Zariski dense in $Aut^0(V,Q)$.
\end{theorem}
Using results of C. Schoen and P. Deligne the above theorem implies:
\begin{corollary}\label{cor:introduction MT group}
Let $\mathcal{M}_{n,2n+2}$ be the coarse moduli space of the CY $n$-folds obtained by a resolution of double covers of $\P^n$ branched along $2n+2$ hyperplanes in general position. Then for $n\geq 3$ the special Mumford-Tate group of a general member in $\mathcal{M}_{n,2n+2}$ is $Aut^0(V,Q)$.
\end{corollary}
This result will imply in turn that any good family for the moduli space $\mathcal{M}_{n,2n+2}$ has the Zarisiki dense monodromy group.

As a consequence of Theorem \ref{thm:introduction Zarisky density}, we can actually show that the fundamental group of the coarse moduli space $\mathfrak{M}_{n,m}, m\geq n+3$ of $m$ ordered hyperplanes in $\P^n$ in general position is large. More precisely, we have
 \begin{corollary}\label{cor:introduction large fundamental group}
  The fundamental group $\pi_1(\mathfrak{M}_{n,m}, s)$ is large, that is, there is a homomorphism of $\pi_1(\mathfrak{M}_{n,m}, s)$ to a noncompact semisimple real algebraic group which has Zariski-dense image.
  \end{corollary}
\begin{proof}
One can deduce easily from Theorem \ref{thm:Zariski density} that the real monodromy representation
$\rho_{\R}: \pi_1(\mathfrak{M}_{AR}, s)\rightarrow Aut(V\otimes\R,Q)$ has Zariski-dense image, and $Aut^0(V\otimes\R,Q)$ is a noncompact semisimple real algebraic group. Hence $\pi_1(\mathfrak{M}_{n,2n+2},s)=\pi_1(\mathfrak{M}_{AR},s)$ is large.

If $m\geq 2n+2$, we consider the configuration space $X(n,m):=\{(H_1,\cdots, H_m)\in (\hat{\P}^n)^m \mid H_1,\cdots, H_m \textmd{ is in general position}\}$, where $\hat{\P}^n$ is the dual projective space of $\P^n$. Obviously $\mathfrak{M}_{n,m}=X(n,m)/PGL(n,\C)$ and the quotient map $X(n,m)\xrightarrow{\pi_{n,m}}\mathfrak{M}_{n,m}$ is a $PGL(n,\C)$ principle bundle. Since $m\geq 2n+2$, we can define the  natural forgetful map $\tilde{f}:X(n,m)\rightarrow X(n,2n+2)$, which sends an ordered hyperplane $(H_1,\cdots, H_m)$ to $(H_1,\cdots, H_{2n+2})$. It can be seen easily that $\tilde{f}$ descends to a forgetful map  $f:\mathfrak{M}_{n,m}\rightarrow\mathfrak{M}_{n,2n+2}$ and we have a commutative diagram
\begin{diagram}
 X(n,m) &\rTo^{\tilde{f}} &X(n,2n+2)\\
\dTo_{\pi_{n,m}}  &    &\dTo_{\pi_{n,2n+2}}\\
 \mathfrak{M}_{n,m}  &\rTo^{f} &\mathfrak{M}_{n,2n+2}
\end{diagram}
By \cite{T-fundamental groups}, Corollary 5.6, $\tilde{f}$ induces a surjective homomorphism from the fundamental group of $X(n,m)$ to the fundamental group of $X(n,2n+2)$. Since $X(n,2n+2)$ is a principle $PGL(n,\C)$ bundle over  $\mathfrak{M}_{n,2n+2}$, the map $\pi_{n,2n+2}$ also induces a surjective map between fundamental groups. Then we deduce the map $\pi_1(\mathfrak{M}_{n,m}, s)\xrightarrow{f_*}\pi_1(\mathfrak{M}_{n,2n+2},f(s))$ is surjective, and  $\pi_1(\mathfrak{M}_{n,m}, f(s))$ is large because of the largeness of $\pi_1(\mathfrak{M}_{n,2n+2}, s)$.

If $m\leq 2n+2$, then $m\geq 2(m-n-2)+2 $. We  have  the association isomorphism (Ch. III, \cite{DO}): $\mathfrak{M}_{n,m}\simeq \mathfrak{M}_{m-n-2,m}$. Then the largeness of $\pi_1(\mathfrak{M}_{n,m}, s)$ follows from the largeness of $\pi_1(\mathfrak{M}_{m-n-2,m}, s)$, which has been verified  in the last paragraph.
\end{proof}
Large groups are infinite and, moreover, always contain a free group of rank two. This corollary can be viewed as a degenerate case of a result of Carlson and Toledo in \cite{C-T}, where they considered the fundamental groups of parameter spaces of hypersurfaces in projective spaces, and showed that except several obvious cases, the kernels of monodromy representations are always large.

A remark on the methodology of the paper before explaining our strategy in detail: A usual method to show the Zariski-density of the monodromy group of a family of algebraic varieties is to show the existence of enough Lefschetz degenerations, which is based on the work of Deligne (see Proposition 5.3, Theorem 5.4 in \cite{D-Weil} I and Lemma 4.4.2 in \cite{D-Weil} II). Instead of seeking for such degenerations towards the boundary of the moduli space \footnote{We conjecture that the moduli space $\mathcal{M}_{n,2n+2}, n\geq 3$ admits no Lefschetz degeneration at all.}, we work on an infinitesimal invariant of the variation of Hodge structure (i.e. the characteristic subvariety) at a general interior point of the moduli space. The general notion of the infinitesimal variation of Hodge structure (abbreviated as IVHS) was first introduced by P. Griffiths and his collaborators as a surrogate of the theta divisor in a Jacobian (see \cite{C-G-G-H}, \cite{Gri} and also Ch. III \cite{Griffiths} for a nice exposition). It has been proven to be very successful in establishing Torelli-type results for algebraic varieties (see e.g. Ch. XII \cite{Griffiths}, \cite{Te}). The result of this paper shows that one can also use IVHS to obtain some important topological assertion on a moduli space of algebraic varieties. In particular, our method can be used to show also the Zariski-density of monodromy groups of good families for moduli spaces of smooth hypersurfaces in projective spaces. Now we proceed to explain the strategy of the proof of our main result Theorem \ref{thm:introduction Zarisky density} in the following three steps. 

\textbf{Step 1}: We show that the moduli space $\mathfrak{M}_{hp}$ of ordered $2n+2$ distinct points on $\P^1$ can be embedded into $\mathfrak{M}_{AR}$, and under this embedding, the induced VHS $\tilde{\V}_{AR}\mid_{\mathfrak{M}_{hp}}$ by restriction is isomorphic to the $n$-th wedge product of the weight one $\Q$-PVHS $\V_C$ associated to a good family of hyperelliptic curves over $\mathfrak{M}_{hp}$.
Then we have a commutative diagram of monodromy representations
\begin{diagram}
 \pi_{1}(\mathfrak{M}_{hp},s) &\rTo^{\tau} &Aut(\V_{C,s}, Q)\\
\dTo  &    &\dTo_{\rho_{\wedge^n}}\\
\pi_1(\mathfrak{M}_{AR}, s)   &\rTo^{\rho} &Aut(V,Q)
\end{diagram}
where $\rho_{\wedge^n}$ is the homomorphism  induced by the $n$-th wedge product of the standard representation of $Aut(\V_{C,s}, Q)=Sp(2n, \Q)$.

By Theorem 1 of \cite{A'Campo},  $\tau(\pi_{1}(\mathfrak{M}_{hp},s))$ is Zariski dense in $Sp(2n, \Q)$. So we
get the commutative diagram of homomorphisms
\begin{diagram}
Sp(2n, \Q)&  &\rTo^{\rho_{\wedge^n}}& &Aut(V,Q)\\
&\rdInto^{}& & \ruInto^{}\\
& &Mon
\end{diagram}

\textbf{Step 2}:
Define the  complex simple Lie algebra:
\begin{displaymath}
\mathfrak{g}_n=\left\{
                 \begin{array}{ll}
                   \mathfrak{sp}_{{2n\choose n}}\C, & \hbox{n odd;} \\
                   \mathfrak{so}_{{2n\choose n}}\C, & \hbox{n even.}
                 \end{array}
               \right.
\end{displaymath}
Then we argue that the following classification result can be applied to the commutative diagram of homomorphisms  in \textbf{Step 1}, so  that  we get either $Mon^0=Aut^0(V,Q)$, or (after a possible finite \'{e}tale base change) there exists a local system of complex vector spaces of rank $2n$ over $\mathfrak{M}_{AR}$, saying $\W$, such that we have an isomorphism of local systems $\tilde{\V}_{AR}\otimes\C\simeq \wedge^n\W$.
\begin{proposition}\label{prop:introduction represnetation theory}
The $n$-th wedge product $V=\wedge^n \C^{2n}$ of the standard representation of $\mathfrak{sp}_{2n}\C$ induces an embedding $\mathfrak{sp}_{2n}\C\hookrightarrow \mathfrak{g}_n$. Suppose $\mathfrak{g}$ is a complex semi-simple Lie algebra lying  between  $\mathfrak{sp}_{2n}\C$ and  $\mathfrak{g}_n$ such that the induced representation of $\mathfrak{g}$ on $V$ is irreducible, then $\mathfrak{g}$ is one of the following:
\begin{itemize}
\item[(1)] $\mathfrak{g}_n$,
\item[(2)] $sl_{2n}\C$, in which case the induced representation of $\mathfrak{g}$ on $V$ is isomorphic to the n-th wedge product of the standard  representation on $\C^{2n}$.
\end{itemize}
\end{proposition}

\textbf{Step 3}: This is the essential step which proves particularly Theorem \ref{thm: introduction not factor canonically}. Indeed, assuming the case $\tilde{\V}_{AR}\otimes\C\simeq \wedge^n\W$, the next proposition will imply this isomorphism is in fact an isomorphism of $\C$-PVHS, for a suitable $\C$-PVHS structure on $\W$, and this would imply that $\tilde{\V}_{AR}$  factors through $\V_{can}$ over $D^I_{n,n}$. A contradiction with Theorem \ref{thm: introduction not factor canonically}. So the only possibility is $Mon^0=Aut^0(V,Q)$ after Proposition \ref{prop:introduction represnetation theory}, and we are done. As already explained above, Theorem \ref{thm: introduction not factor canonically} will be achieved by showing that the characteristic subvarieties attached to $\tilde{\V}_{AR}$ and $\V_{can}$ are non-isomorphic.
\begin{proposition}\label{prop:introduction decomposition of PVHS}
Let $\bar{S}$ denote a projective manifold, $Z$ a simple divisor with normal crossing and $S=\bar{S}\backslash Z$. Let $\V$ denote an irreducible $\C$-PVHS over $S$ with quasi-unipotent local monodromy around each component of $Z$. Suppose $\W$ is a rank $2n$ local system over $S$ and we have an isomorphism  of local systems $\V\simeq \wedge^n\W$, then $\W$ admits the structure of a $\C$-PVHS such that the induced $\C$-PVHS on the wedge product $\wedge^n\W$ coincides with the given $\C$-PVHS on $\V$.
\end{proposition}

\section{Calabi-Yau manifolds coming from hyperplane arrangements}
Throughout this paper we use the following notation:

Let $M$ be a $\C$-linear space, or a sheaf of  $\C$-linear spaces on a scheme, on which the group  $\Z/r\Z=<\sigma>$ acts. Let $\zeta$ be a primitive $r$-th root of unit. For $i\in \Z/r\Z$ we write
$M_{(i)} :=\{x\in M\mid \sigma(x)=\zeta^i x\}$,
which in the sheaf case has to be interpreted on the level of local sections. We
refer to $M_{(i)}$ as the $i$-eigenspace of $M$. We have $M = \oplus_{i\in \Z/r\Z}M_{(i)}$.

An ordered arrangement $\mathfrak{A}=(H_1,\cdots, H_{2n+2})$ of $2n+2$ hyperplanes in $\P^n$ can be given by a matrix $A\in M((n+1)\times (2n+2), \C)$, the $j$th column corresponding to the defining equation
$$
\sum_{i=0}^na_{ij}x_i=0
$$
of the hyperplane $H_j$. Here $[x_0: \cdots :x_n]$ are the homogeneous coordinates on $\P^n$. We say that $\mathfrak{A}$ is in general position if no $n+1$ of the hyperplanes intersect in a point. In terms of the matrix $A$ this means that each $(n+1)\times (n+1)$-minor is non-zero.

\subsection{The double cover of $\P^n$ and its crepant resolution}\label{subsection:double cover and crepant resolution}

The hyperplanes of the arrangement $\mathfrak{A}$ determine a divisor $H=\sum_{i=1}^{2n+2}H_i$ on $\P^n$. As the degree of $H$ is even and the Picard group of $\P^n$ has no torsion, there exists a unique double cover $\pi: X\rightarrow \P^n$ that ramifies over $R$. By a direct computation using the adjunction formula, the canonical line bundle of $X$ is trivial. The singular locus of such a double cover $X$ is precisely the preimage of the singular locus of $H$.
Fix an order of irreducible components of singularities of $H$, say
$Z_1,\cdots,Z_{N}$. The canonical resolution of $X$ according to that order is the following commutative diagram:\\

\begin{displaymath}
\begin{diagram}[labelstyle=\scriptscriptstyle]
X        & = &X_0        &\lTo^{\tau_1}  &X_1         &\lTo^{\tau_2}  &\cdots &\lTo^{\tau_{N}}&X_{N} & =  & \tilde{X}    \\
\dTo^{\pi} & & \dTo_{\pi_0}&               &\dTo_{\pi_1}&               &       &            & \dTo_{\pi_N} & & \dTo_{\tilde{\pi}}\\
\P^n     & = &\P_0        &\lTo^{\sigma_1}&\P_1        &\lTo^{\sigma_2}&\cdots &\lTo^{\sigma_{N}}&\P_{N} & = & \tilde{\P^n}\\
\end{diagram}
\end{displaymath}
Here, inductively on $i$, $\P_i\stackrel{\sigma_i}{\to} \P_{i-1}$ is the blow-up of
$\P_{i-1}$ along the smooth center
$(\sigma_{i-1}\circ\cdots\circ\sigma_1)^{-1}(Z_i)$, and $X_i$ is
the normalization of the fiber product of $X_{i-1}$ and $\P_i$ over
$\P_{i-1}$.

\begin{lemma}\label{lemma:infinitesimal deformation of X}
The space of infinitesimal deformations of $\tilde{X}$ is naturally isomorphic to the space of infinitesimal deformations of $\mathfrak{A}$.
\end{lemma}
\begin{proof}
For the $n=3$ case, see Lemma 2.1 in \cite{GSSZ}, whose proof is based on \cite{Cynk-van Straten}. This proof goes through in general case verbatim.
\end{proof}

\begin{proposition}\label{hodge number}
Let $X$ and $\tilde X$ be as above. Then
$$
\dim H_{prim}^{p,q}(\tilde X)={n \choose p}^2,\quad p+q=n.
$$
\end{proposition}
The proof of this proposition is postponed to Section \ref{subsec:Hodge number calculation}.

Throughout this paper except Section \ref{section:general results}  $\mathfrak{M}_{AR}$ is denoted for the coarse moduli space of ordered arrangements of $2n+2$ hyperplanes in $\P^n$ in general position. Let $\mathcal{M}_{n,2n+2}$ denote the coarse moduli space of $\tilde{X}$ and  call $\mathcal{M}_{n,2n+2}$ the Dolgachev's moduli spaces.

Let $\mathfrak{A}$ be an ordered arrangement in general position. It is easy to verify that under the automorphism group of $\P^n$ one can transform in a unique way the ordered first  $n+2$ hyperplanes $(H_1,\cdots, H_{n+2})$ of $\mathfrak{A}$ into the ordered $n+2$ hyperplanes in $\P^n$, that are given by the first $n+2$ columns in the following matrix $A$. Hence the moduli point of $\mathfrak{A}$ in $\mathfrak{M}_{AR}$ can be uniquely represented by the matrix $A$ of the form:
$$
\left(
  \begin{array}{cccccccc}
    1 & 0 & \cdots & 0 & 1 & 1 & \cdots & 1 \\
    0 & 1 &  & 0 & 1 & a_{11} & \cdots & a_{1n} \\
    \vdots &  & \ddots & \vdots & \vdots & \vdots &  & \vdots \\
    0 &  &  & 1 & 1 & a_{n1} & \cdots & a_{nn} \\
  \end{array}
\right)
$$
Conversely a matrix $A$ in the above form whose all $(n+1)\times (n+1)$-minors are non-zero represents an arrangement $\mathfrak{A}$ in general position. Thus
$\mathfrak{M}_{AR}$ can be realized as an open subvariety of the affine space $\C^{n^2}$ and it admits a natural family $f:\mathcal{X}_{AR}\rightarrow \mathfrak{M}_{AR}$, where each fiber $f^{-1}(\mathfrak{A})$ is the double cover of $\P^n$ branched along the hyperplane arrangement $\mathfrak{A}\in \mathfrak{M}_{AR}$.
It is easy to see the canonical resolution  gives rise to a simultaneous resolution $\tilde{\mathcal{X}}_{AR}\rightarrow \mathcal{X}_{AR}$, and the family $\tilde{f}:\tilde{\mathcal{X}}_{AR}\rightarrow \mathfrak{M}_{AR}$ is a smooth family of CY manifolds.
\subsection{The Kummer cover}\label{subsection:Kummer cover}
Let $a=(a_{ij})$ denote a $n\times n$-matrix associated with  a hyperplane arrangement $\mathfrak{A}$ as described above. We let $Y$ denote the complete intersection of the $n+1$ hypersurfaces in $\P^{2n+1}$ defined by the $n+1$ equations:
\begin{equation}\notag
\begin{split}
&y_{n+1}^2-(y_0^2+\cdots+y_n^2)=0;\\
& y_{n+i+1}^2-(y_0^2+a_{1i}y_1^2+\cdots +a_{ni}y_n^2)=0, \ 1\leq i\leq   n.
\end{split}
\end{equation}
Here $[y_0:\cdots:y_{2n+1}]$ are the homogeneous coordinates on $\P^{2n+1}$. In case $\mathfrak{A}$ is in general position, the space $Y$ is smooth (See Proposition 3.1.2 in \cite{Te}).

Let $N=\oplus_{j=0}^{2n+1}\F_2$, where $\F_2=\Z/2\Z$. Consider the  following group
\begin{equation}\notag
\begin{split}
N_1:=Ker(N &\rightarrow \F_2)\\
(a_j)&\mapsto \sum_{j=0}^{2n+1}a_j
\end{split}
\end{equation}
We define a natural  action of $N$ on $Y$.  $\forall g=(a_0,\cdots, a_{2n+1})\in N$, the action of $g$ on $Y$ is induced by
\begin{equation}\notag
g\cdot y_j :=(-1)^{a_j}y_j, \ \ \forall \ 0\leq j\leq 2n+1.
\end{equation}
Then for a given arrangement $\mathfrak{A}$ in general position, we have the following relations between the double cover $X$ and the Kummer cover $Y$:
\begin{proposition}\label{prop:pure Hodge structure of X}
\begin{itemize}
\item[(1)] The map $\pi_1: Y\rightarrow \P^n$, $[y_0:\cdots:y_{2n+1}]\mapsto [y_0^2:\cdots:y_{n}^2]$ defines a cover of degree $2^{2n+1}$.
\item[(2)]$X\simeq Y/N_1$.
\item[(3)] There exists a natural isomorphism of rational Hodge structures $H^n(X,
\Q)\simeq H^n(Y,\Q)^{N_1}$, where $H^n(Y,\Q)^{N_1}$ denotes the subspace of invariants under $N_1$. In particular, the natural mixed Hodge structure on $H^n(X,
\Q)$ is in fact a pure one.
\item[(4)]Via the identification $H^n(X,\Q)= H^n(Y,\Q)^{N_1}$, we have $H^n(X,\Q)_{(1)}\subset H^n(Y,\Q)_{pr}$.
\end{itemize}
\end{proposition}
\begin{proof}
For (1), (2), (3), See Lemma 2.4, Proposition 2.5 and Proposition 2.6 in \cite{GSSZ}.

(4) Since the isomorphism $X\simeq Y/N_1$ is compatible with the $\F_2$-action, we know $H^{n-2}(X,\Q)_{(1)}=H^{n-2}(Y,\Q)^{N_1}_{(1)}$, where the $\F_2$-action on $Y/N_1$ is induced by the identification  $\F_2=N/N_1$. Since $Y$ is a complete intersection in $\P^{2n+1}$, Lefschetz Hyperplane Theorem implies that
$H^{n-2}(Y,\Q)=H^{n-2}(\P^{2n+1},\Q)$, hence $H^{n-2}(X,\Q)_{(1)}=H^{n-2}(Y,\Q)^{N_1}_{(1)}=0$. Since the Lefschetz operator
$$
H^{n-2}(X,\Q)\xrightarrow{L}H^{n}(X,\Q)
$$
 preserves the $\F_2$-action, we get $L(H^{n-2}(X,\Q))\cap H^n(X,\Q)_{(1)}=0$. This shows  $H^n(X,\Q)_{(1)}\subset H^n(Y,\Q)_{pr}$.
\end{proof}

\subsection{Hyperelliptic locus}\label{subsection:hyperelliptic locus}
There is an interesting locus in $\mathfrak{M}_{AR}$ where the Hodge structure of $X$ comes from a hyperelliptic curve.

Note that there exists a natural Galois covering with Galois group $S_n$, the permutation group of $n$ letters:
$$
\gamma: (\P^1)^n\rightarrow Sym^n(\P^1)=\P^n.
$$
Here the identification attaches to a divisor of degree $n$ the ray of its equation in $H^0(\P^1, \mathcal{O}(n))$.

\begin{lemma}\label{lemma:points in general position implies hyperplane in general position}
Let $(p_1,\cdots, p_{2n+2})$ be a collection of $2n+2$ distinct points on $\P^1$, and put $H_i=\gamma(\{p_i\}\times \P^1\times \cdots \times \P^1)$. Then $(H_1,\cdots, H_{2n+2})$ is an arrangement of hyperplanes in general position.
\end{lemma}

\begin{proof}
The divisors of degree $n$ in $\P^1$ containing a given point form a hyperplane and,
as a divisor of degree $n$ cannot contain $n+1$ distinct points, no $n+1$ hyperplanes in
the arrangement do meet.
\end{proof}

Let $C$ be the double cover of $\P^1$ branched at $p_1,\cdots, p_{2n+2}$, and let $X_C$ be the double cover of $\P^n$ branched along $H_1,\cdots, H_{2n+2}$. Then the double covering structures induce natural actions of  the cyclic group $\F_2=\Z/2\Z$  on $C$ and $X$.

The group $\F_2^n$ and the permutation group $S_n$ act naturally on the product $C^n$. These actions induce an action of the semi-direct product $\F_2^n\rtimes S_n$ on $C^n$. Let $N^{'}$ be the kernel of the summation homomorphism:
\begin{equation}\notag
\begin{split}
\F_2^n &\rightarrow \F_2 \\
(a_i)&\mapsto \sum_{i=1}^{n}a_i
\end{split}
\end{equation}
Then we have:
\begin{lemma}\label{lemma: X_C simeq C^n/N_2}
There exists a natural isomorphism: $X_C \simeq C^n/N_2$, where $N_2:=N^{'}\rtimes S_n$.
\end{lemma}

\begin{proof}
Let $p: C\rightarrow \P^1$ be the covering map. The $n$-fold product
$$
h : C^n\xrightarrow{p^n}(\P^1)^n\xrightarrow{\gamma}\P^n
$$
is a Galois cover with Galois group $\F_2^n\rtimes S_n$.
Similar to \cite{GSSZ}, Lemma 2.8,
 one checks that the natural map $C^n/N_2\rightarrow \P^n$ induced by $h$ is a double
cover and its branch locus is exactly $H_1+\cdots+H_{2n+2}$. As the Picard
group of a projective space has no torsion, one concludes that $C^n/N_2$ is isomorphic to
the double cover $X_C$ of $\P^n$ branched along $\sum_{i}H_i$.
\end{proof}

%
  The following lemma is well known
\begin{lemma}\label{lemma: eigen Hodge numbers of C}
 $dim H^{1}(C,\C)_{(0)}=0$, $dim H^{1,0}(C)_{(1)}= $ $dim H^{0,1}(C)_{(1)}=n$.
\end{lemma}
\begin{proof}
See \cite{M}, (2.7).
\end{proof}

We can get some information about the Hodge structure of $X_C$ from $C$.

\begin{proposition}\label{prop:Hodge structure of X_C}

$$
H^n(X_C,\C)_{(1)}\simeq \wedge^n H^1(C,\C).
$$

\end{proposition}

\begin{proof}
Lemma \ref{lemma: X_C simeq C^n/N_2}  gives an identification
$$
H^n(X_C,\C)=H^n(C^n,\C)^{N_2}
$$
Taking into account the action of $\F_2$,  the K\"{u}nneth formula gives
$$
H^n(C^n,\C)^{N_2}_{(1)}=(\otimes_{i=1}^n H^1(C,\C))^{N_2}
$$
Then it is easy to see the following map is an isomorphism
\begin{equation}\notag
\begin{split}
\wedge^n H^1(C,\C) &\xrightarrow{\sim} (\otimes_{i=1}^n H^1(C,\C))^{N_2}=H^n(X_C,\C)_{(1)}\\
\alpha_1\wedge \cdots \wedge \alpha_n &\mapsto \sum_{\nu\in S_n}(-1)^{Sign(\nu)} \alpha_{\nu(1)}\otimes \cdots \otimes \alpha_{\nu(n)}
\end{split}
\end{equation}
and this gives the desired isomorphism.


\end{proof}

Let $\mathfrak{M}_{hp}\subset \mathfrak{M}_{AR}$ be the moduli space of ordered distinct $2n+2$ points on $\P^1$ and  $g:\mathcal{C}\rightarrow \mathfrak{M}_{hp}$ be the natural universal family of hyperelliptic curves.
 Recall that $f: \mathcal{X}_{AR}\rightarrow \mathfrak{M}_{AR}$ is the natural family of double covers of $\P^n$ branched along $2n+2$ hyperplanes in general position and $\tilde{f}:\tilde{\mathcal{X}}_{AR}\rightarrow \mathfrak{M}_{AR}$  is the simultaneous crepant resolution of $f$. Consider the $\Q-$VHS (rational variation of Hodge structures) $\V:=R^nf_{*}\Q$. Since $\F_2$ acts naturally on $\V$, We have a decomposition of $\Q-$VHS: $\V =\V_{(0)}\oplus\V_{(1)}$. Similarly, let $\tilde{\V}_{AR}:=(R^n\tilde{f}_{*}\Q)_{pr}$ be the $\Q-$PVHS of primitive cohomologies. Then we have

\begin{proposition}\label{prop:reduce to V_1}
There are isomorphisms of
 $\Q-$PVHS:
 \begin{itemize}
 \item[(1)] $\tilde{\V}_{AR}\simeq \V_{(1)}$, whose Hodge numbers are:
$$
h^{p,n-p}={n\choose p}^2, \ \forall \ 0\leq p\leq n.
$$
\item[(2)]$\V_{(1)}\mid_{\mathfrak{M}_{hp}}\simeq \wedge^n\V_{C}$, where  $\V_C:=R^1g_{*}\Q$ is the weight one $\Q$-PVHS assocaited to $g:\mathcal{C}\rightarrow \mathfrak{M}_{hp}$.
\end{itemize}
\end{proposition}
\begin{proof}
(1) By Proposition \ref{prop:pure Hodge structure of X}, we can see that the simultaneous crepant resolution gives  a natural morphism $\V_{(1)}\rightarrow \tilde{\V}_{AR}$ of $\Q$-PVHS. Theorem 5.41 in \cite{PS} shows that this morphism is injective.
 It is easy to see the isomorphism in Proposition \ref{prop:Hodge structure of X_C} preserves the Hodge filtrations. Then Proposition \ref{hodge number} and Lemma \ref{lemma: eigen Hodge numbers of C} show that $\tilde{\V}_{AR}$ and $\V_{(1)}$ have the same Hodge numbers. So the natural morphism $\V_{(1)}\rightarrow \tilde{\V}_{AR}$ is an isomorphism of $\Q$-PVHS.

 (2) follows directly from Proposition \ref{prop:Hodge structure of X_C}.
\end{proof}

\section{Type A canonical variation and characteristic subvariety}


In order to compare two $\C$-PVHSs over $S$, besides using the obvious Hodge numbers as invariants, we can also use another important series of invariants: characteristic subvarieties, which  are contained in the projectivized tangent bundle $\P(\mathcal{T}_S)$. The basic theory of characteristic subvarieties is developed in \cite{SZ}. We recall the definition.
\begin{definition}
Let $\W$ be a $\C$-PVHS of weight $n$ over $S$ and $(E,\theta)$ the associated Higgs bundle. For every $q$ with $1\leq q\leq n$,  the $q$th iterated Higgs field
$$
E^{n,0}\xrightarrow{\theta^{n,0}}E^{n-1,1}\otimes \Omega_{S}\xrightarrow{\theta^{n-1,1}}\cdots \xrightarrow{\theta^{n-q+1,q-1}}E^{n-q,q}\otimes S^{q}\Omega_S
$$
 defines a morphism
$$
\theta^q: Sym^q(T_S)\rightarrow Hom(E^{n,0}, E^{n-q,q}).
$$
Then $\forall s\in S$, the $q$th characteristic subvariety of $\W$ at $s$ is
$$
C_{q,s}=\{[v]\in \P(T_{S,s})\mid v^{q+1}\in ker(\theta^{q+1})\}.
$$
Moreover, we call $\W$ of Calabi-Yau type (CY-type) if rank $E^{n,0}=1$ and $\theta^{n,0}: T_S\xrightarrow{} Hom(E^{n,0}, E^{n-1,1})$ is an isomorphism at the generic point.
\end{definition}
\begin{remark}
Our definition of characteristic variety is slightly different with that in \cite{SZ}. Here we only need the reduced subvariety. See Lemma 4.3 in \cite{GSSZ}.
\end{remark}

As an example, we can compute the first characteristic subvariety of the type A canonical $\C$-PVHS explicitly.
We recall the following description of  the canonical $\C$-PVHS $\V_{can}$ over $D^{I}_{n,n}$ from \cite{SZ}.

Let $V=\C^{2n}$ be a complex vector space equipped with a Hermitian symmetric bilinear form $h$ of signature $(n,n)$. Then $D^{I}_{n,n}$ parameterizes the dimension $n$ complex subspaces $U\subset V$ such that
$$
h\mid_{U}: U\times U\rightarrow \C
$$
is positive definite. This forms the tautological subbundle $S\subset V\times D^{I}_{n,n}$ of rank $n$ and denote by $Q$ the tautological quotient bundle of rank $n$. We have the natural isomorphism of holomorphic vector bundles
$$
T_{D^I_{n,n}}\simeq Hom(S,Q).
$$
The standard representation $V$ of $SU(n,n)$ gives rise to a weight one $\C$-PVHS $\W$ over $D^{I}_{n,n}$, and its associated Higgs bundle
$$
F=F^{1,0}\oplus F^{0,1}, \ \ \eta=\eta^{1,0}\oplus \eta^{0,1}
$$
is determined by
$$
F^{1,0}=S, \ \ F^{0,1}=Q, \ \ \eta^{0,1}=0,
$$
and $\eta^{1,0}$ is defined by the above isomorphism. The canonical $\C$-PVHS is
$$
\V_{can}=\wedge^{n}\W
$$
and its associated system of Higgs bundle $(E_{can},\theta_{can})$ is then
$$
(E_{can}, \theta_{can})=\wedge^n(F,\eta).
$$

\begin{proposition}\label{propo:Hodge numbers and chara. variety of canonical VHS}

\begin{itemize}
\item[(1)]The Hodge numbers  of the type A canonical $\C$-PVHS $\V_{can}$ over $D^I_{n,n}$ are
$$
h^{p,n-p}={n\choose p}^2, \ \ \forall \ 0\leq p\leq n.
$$
\item[(2)]The first characteristic subvarieties of the type A canonical $\C$-PVHS $\V_{can}$ over $D^I_{n,n}$ are
    $$
    C_{1,s}\simeq \P^{n-1}\times \P^{n-1}, \ \ \forall s\in D^{I}_{n,n}.
    $$
\end{itemize}
\end{proposition}
\begin{proof}
From the description of $\V_{can}$ above, $\forall \ 0\leq p\leq n$, $E_{can}^{p,n-p}$ can be identified with the subspace of $\wedge^nF$ linearly spanned by the following set:
$$
\{e_1\wedge\cdots \wedge e_p\wedge \tilde{e}_1\wedge \cdots \wedge \tilde{e}_{n-p}\mid e_i\in F^{1,0}, 1\leq i\leq p; \  \tilde{e}_j\in F^{0,1}, 1\leq j\leq n-p.\}
$$
From this we get (1).
Moreover, with this identification, $\forall \ v \in T_{D_{n,n}^I}$, if $e_1,\cdots, e_n \in F^{1,0}$, then
$$
\theta_{can,v}^{n,0}(e_1\wedge\cdots \wedge e_n)=\sum_{i=1}^n (-1)^{i-1}\eta^{1,0}_{v}(e_i)\wedge e_1\wedge\cdots\wedge \hat{e_i}\wedge \cdots \wedge e_n,
$$
where $\hat{e_i}$ means deleting the term $e_i$.

Similarly, if $e_1, \cdots, e_{n-1} \in F^{1,0}$, $\tilde{e}_{1}\in F^{0,1}$, then
$$
\theta_{can, v}^{n-1,1}(e_1\wedge\cdots \wedge e_{n-1}\wedge \tilde{e}_1)=\sum_{i=1}^{n-1}(-1)^{i-1}\eta^{1,0}_{v}(e_i)\wedge e_1\wedge\cdots\wedge \hat{e_i}\wedge \cdots \wedge e_{n-1}\wedge\tilde{e}_1.
$$

Since $\eta^{1,0}$ is defined by the identification $T_{D^I_{n,n}}=Hom (F^{1,0}, F^{0,1})$, we can see that $\forall \ s\in D^I_{n,n}$,
\begin{equation}\notag
\begin{split}
C_{1,s}&= \{v\in T_{D^I_{n,n},s}\mid \theta^{n-1,1}_{can,v}\circ \theta^{n,0}_{can,v}=0.\}\\
       &\simeq\{A\in M(n\times n, \C)\mid rank(A)=1.\}\\
       &\simeq \P^{n-1}\times \P^{n-1}.
\end{split}
\end{equation}
This proves (2).
\end{proof}

\section{Characteristic subvariety and nonfactorization}

\subsection{Nonfactorization of the period map}

There is  an upper bound of the dimension of the first characteristic variety of the  $\Q$-PVHS $\V_{(1)}$.
\begin{proposition}\label{prop:dim of characteristic variety}
If $n\geq 3$ ,  then
for a  generic $a\in \mathfrak{M}_{AR}$,   the first characteristic variety of $\V_{(1)}$ has dimension $dim C_{1,a}\leq 2$.
\end{proposition}
\begin{proof}
This theorem follows from Proposition \ref{prop:Jacobian ring: dim upper bound}, whose proof we postpone to the section \ref{subsection:Jacobian ring}.
\end{proof}

As over a coarse moduli space $\mathfrak{M}$ of a polarized algebraic variety  usually  does not exist a universal family, we use the following weaker notion. We say that a proper smooth morphism $f: \mathcal{X}\rightarrow S$ over a smooth connected base is a good family for $\mathfrak{M}$, if the moduli map $S\rightarrow \mathfrak{M}$ is dominant and generically finite. With this definition, our main theorem is
\begin{theorem}\label{thm: does not factor canonically}
If $n\geq 3$,  then
\begin{itemize}
\item[(1)] $\tilde{\mathcal{X}}_{AR}\xrightarrow{\tilde{f}}\mathfrak{M}_{AR}$ is a good family for the coarse moduli space  $\mathcal{M}_{n,2n+2}$;
\item[(2)] Let $f: \mathcal{X}\rightarrow S$ be a good family of $\mathcal{M}_{n,2n+2}$ and $\W=(R^nf_{*}\Q)_{pr}$ be the associated weight $n$ $\Q$-PVHS. Then $\W$   does not factor through the  $\C$-PVHS $\V_{can}$ over the type $A$ symmetric domain $D_{n,n}^I$.
    \end{itemize}
\end{theorem}
\begin{proof}
(1) follows directly from Lemma \ref{lemma:infinitesimal deformation of X}.

(2) Assume the contrary. By Proposition \ref{propo:Hodge numbers and chara. variety of canonical VHS}, for any $s\in S$ away from the ramification locus of the moduli map $S\rightarrow \mathcal{M}_{n,2n+2}$, the first characteristic variety $C_{1,s}$ is isomorphic to $\P^{n-1}\times \P^{n-1}$. In particular, the dimension of $C_{1,s}$ has dimension $2n-2\geq 4$. On the other hand, by Proposition \ref{prop:reduce to V_1}, Proposition \ref{prop:dim of characteristic variety} and (1), for a generic $s\in \mathcal{X}$, the dimension of the first characteristic variety $C_{1,s}$ has dimension $\leq 2$. This gives a contradiction.
\end{proof}


\section{Zariski density of monodromy group}
In this section, we will use the notations as given in \cite{FH} for representations of Lie algebras. We are mainly concerned with  representations of $\mathfrak{sp}_{2n}\C$. Following the notations in Section 16.1 in \cite{FH}, it is well known  that the weight lattice of $\mathfrak{sp}_{2n}\C$ is the lattice of integral linear combinations of $L_1,\cdots, L_n \in \mathfrak{h}^{*}$, where $\mathfrak{h}\subset \mathfrak{sp}_{2n}\C$ is the Cartan subalgebra as defined  in Section 16.1 in \cite{FH}, and after fixing a positive direction, the corresponding closed Weyl chamber is
$$
\{a_1L_1+a_2L_2+\cdots+a_nL_n: a_1\geq a_2\geq \cdots \geq a_n\geq 0\}.
$$

We have the following
\begin{lemma}\label{lemma1:representation theory}
Suppose $W$ is a nontrivial finite dimensional complex representation of $\mathfrak{sp}_{2n}\C$ and $Q$ is a  $\mathfrak{sp}_{2n}\C$-invariant bilinear form on W. Given a partition $\lambda=(\lambda_1,\cdots, \lambda_k)$ of an integer $d$ with $k\leq dim W-1$, then there exist $k$ linearly independent weight vectors $v_1,\cdots, v_k \in W$ such that if we define $v_{\lambda}\in W^{\otimes d}$ as follows:
$$
\begin{array}{ccc}
  \underbrace{v_1 \otimes \cdots\otimes v_1} & \otimes\cdots\otimes & \underbrace{v_k\otimes\cdots\otimes v_k} \\
  \lambda_1 &  & \lambda_k
\end{array}
$$
then $v_{\lambda}\cdot c_{\lambda}\in \mathbb{S}_{\lambda}W$ is a nonzero weight vector, and the weight $a_1L_1+\cdots +a_nL_n$ of $v_{\lambda}\cdot c_{\lambda}$ satisfies $a_1\geq \lambda_1$. Here $c_{\lambda}$ is the Young symmetrizer. Moreover, if $k\leq [\frac{dim W}{2}]$, then the weight vectors $v_1,\cdots, v_k \in W$ above can be chosen such that they satisfy the additional condition: $\forall \ 1\leq i,j\leq k$, $Q(v_i,v_j)=0$.
\end{lemma}
\begin{proof}
From the definition of the Young symmetrizer, $v_{\lambda}\cdot c_{\lambda}\in \mathbb{S}_{\lambda}W$ is  nonzero as long as $v_1,\cdots, v_k \in W$ are linearly independent. Suppose
$$
W=\oplus_{\alpha\in\mathfrak{h}^{*}}W_{\alpha}
$$
is the weight space decomposition of $W$. It follows from the basic representation theory of $\mathfrak{sp}_{2n}\C$ that $dim W_{\alpha}=dim W_{-\alpha}$, $\forall \ \alpha\in \mathfrak{h}^{*}$. So we can write
$$
W=\bigoplus_{i=1}^{l}(W_{\alpha_i}\oplus W_{-\alpha_i})\oplus U
$$
such that
 \begin{itemize}
 \item[(1)]$U$ is a trivial representation of $\mathfrak{sp}_{2n}\C$;
  \item[(2)]$\forall \ 1\leq i\leq l$, $W_{\alpha_i}\neq 0$;
  \item[(3)]if $\alpha_i=a_{i1}L_1+\cdots +a_{in}L_n $, then there exists an integer $k_i$ between $1$ and $n$, with $ a_{ik_i}> 0 $ and $a_{i1}=a_{i2}=\cdots= a_{i,k_i-1}=0$.
  \end{itemize}
  Since $Q$ is $\mathfrak{sp}_{2n}\C$-invariant, the spaces $U$ and $\bigoplus_{i=1}^{l}(W_{\alpha_i}\oplus W_{-\alpha_i})$ are orthogonal under $Q$, and $Q\mid_{\oplus_{i=1}^{l}W_{\alpha_i}}=0$. It is easy to see that if $v_i\in W_{\beta_i}$, $i=1,\cdots, k$, then $v_{\lambda}\cdot c_{\lambda}$ is a weight vector of $\mathbb{S}_{\lambda}W$ with weight $\sum_{i=1}^{k}\beta_i$. From these and since at least one $a_{i1}$ is a positive integer, it is not difficult to see that we can choose linearly independent weight vectors $v_1,\cdots, v_k \in W$ such that the weight of $v_{\lambda}\cdot c_{\lambda}$ satisfies the required condition. Moreover, if $k\leq [\frac{dim W}{2}]$, since $Q\mid_{\oplus_{i=1}^{l}W_{\alpha_i}}=0$ and $U\perp\oplus_{i=1}^{l}W_{\alpha_i}$ under $Q$,  we can choose weight vectors $v_1,\cdots, v_k$  from the spaces $U$ and $\oplus_{i=1}^{l}W_{\alpha_i}$ such that they satisfy the  additional condition: $\forall \ 1\leq i,j\leq k$, $Q(v_i,v_j)=0$.
\end{proof}


Recall the definition of the complex Lie algebra $\mathfrak{g}_n$ in Section \ref{section:introduction}.
\begin{proposition}\label{prop:represnetation theory}
The $n$-th wedge product $V=\wedge^n \C^{2n}$ of the standard representation of $\mathfrak{sp}_{2n}\C$ induces an embedding $\mathfrak{sp}_{2n}\C\hookrightarrow \mathfrak{g}_n$. Suppose $\mathfrak{g}$ is a complex semi-simple Lie algebra lying  between  $\mathfrak{sp}_{2n}\C$ and  $\mathfrak{g}_n$ such that the induced representation of $\mathfrak{g}$ on $V$ is irreducible, then $\mathfrak{g}$ is one of the following:
\begin{itemize}
\item[(1)] $\mathfrak{g}_n$,
\item[(2)] $sl_{2n}\C$, in which case the induced representation of $\mathfrak{g}$ on $V$ is isomorphic to the n-th wedge product of the standard  representation on $\C^{2n}$.
\end{itemize}
\end{proposition}
\begin{proof}


We first show that $\mathfrak{g}$ is simple. Since $\mathfrak{g}$ is semi-simple, we can write $\mathfrak{g}=\oplus_{i=1}^m\mathfrak{g}_i$ into a direct sum of simple Lie algebras. By Schur's lemma and since $V$ is an irreducible $\mathfrak{g}$-module, we have the tensor decomposition of $V=\otimes_{i=1}^m V_i$, where each $V_i$ is an irreducible $\mathfrak{g}_i$-module. Since $\mathfrak{sp}_{2n}\C$ is simple, $\forall \ 1\leq i\leq m$, the composition
$$
\mathfrak{sp}_{2n}\C\hookrightarrow \mathfrak{g}=\oplus_{i=1}^m\mathfrak{g}_i \rightarrow \mathfrak{g}_i
$$
is either an embedding or a zero map. So $\wedge^n \C^{2n}=V=\otimes_{i=1}^m V_i$ is also a tensor decomposition of the  $\mathfrak{sp}_{2n}\C$-module $\wedge^n \C^{2n}$. A direct  computation of highest weights shows that $m=1$, which implies $\mathfrak{g}$ is simple. So the only possibilities  of $\mathfrak{g}$ are exceptional Lie algebras, of type A, type B, type C or type D.

\textbf{Exceptional cases}: Since $V$ is an irreducible $\mathfrak{g}$-module, by checking the dimensions of irreducible representations, we can exclude the exceptional Lie algebra  cases.

\textbf{Case A}: Suppose $\mathfrak{g}\simeq \mathfrak{sl}_m\C$ is a simple Lie algebra of type A, then by Weyl's construction, there exists a partition $\lambda=(\lambda_1,\cdots, \lambda_k)$ of a positive integer $d$ with $k\leq m-1$, such that $V\simeq \mathbb{S}_{\lambda}W$ as $\mathfrak{sl}_m\C$-modules, where $W=\C^{m}$ is the standard representation of $ \mathfrak{sl}_m\C$ (cf. Proposition 15.15 in \cite{FH}). By Lemma \ref{lemma1:representation theory}, we can always choose $k$ linearly independent vectors $v_1, \cdots, v_k \in W$ such that $v_{\lambda}\cdot c_{\lambda}\in \mathbb{S}_{\lambda}W $  is a nonzero weight vector of $\mathfrak{sp}_{2n}\C$  and  its weight $a_1L_1+\cdots +a_nL_n$ satisfies  $a_1\geq \lambda_1$. Since $V\simeq \mathbb{S}_{\lambda}W $ as $\mathfrak{sp}_{2n}\C$-modules and the highest weight of $V$ is $L_1+\cdots+L_n$, we deduce that $\lambda_1=1$. So the partition $\lambda=(1,\cdots,1)$ and $\mathbb{S}_{\lambda}W=\wedge^d W$. Then it is easy to  see that the only possibility is $d=n$, and $m=2n$. This gives case (2).

\textbf{Case B}: Suppose $\mathfrak{g}\simeq \mathfrak{so}_{2m+1}\C$ is a simple Lie algebra of type B. Following Section 18.1 in \cite{FH}, we choose the  basis $L^{'}_1, \cdots, L^{'}_{m}$ of the dual space of a Cartan subalgebra (note we add a '$'$' to distinguish with the weights of $\mathfrak{sp}_{2n}\C$).  We have two cases to discuss.

 \textbf{B1}: If the highest weight of the irreducible $\mathfrak{so}_{2m+1}\C$-representation $V$ is $a_1L^{'}_1+\cdots+a_{m}L^{'}_m$, with each $a_i$  an integer, then by Weyl's construction, there exists a partition $\lambda=(\lambda_1,\cdots, \lambda_k)$ of a positive integer $d$ with $k\leq m$, such that $V\simeq \mathbb{S}_{[\lambda]}W$ as $\mathfrak{so}_{2m+1}\C$-modules. Here $W=\C^{2m+1}$ is the standard representation of $\mathfrak{so}_{2m+1}\C$. (cf. Theorem 19.22 in \cite{FH}). By Lemma \ref{lemma1:representation theory}, we can  choose $k$ linearly independent vectors $v_1, \cdots, v_k \in W$ such that $v_{\lambda}\cdot c_{\lambda}\in \mathbb{S}_{[\lambda]}W $  is  a nonzero weight vector of $\mathfrak{sp}_{2n}\C$ and its weight $b_1L_1+\cdots +b_nL_n$ satisfies  $b_1\geq \lambda_1$. Since $ V\simeq \mathbb{S}_{\lambda}W $ as $\mathfrak{sp}_{2n}\C$-modules and the highest weight of $V$ is $L_1+\cdots+L_n$, we deduce that $\lambda_1=1$. So the partition $\lambda=(1,\cdots,1)$ and $V=\mathbb{S}_{[(1, \cdots, 1)]} W$. Then it is easy to  see that the only possibility is $d=1$, $2m+1={2n\choose n}$, and $\mathfrak{g}=\mathfrak{g}_n$. This gives case (1).

 \textbf{B2}: If the highest weight of the irreducible $\mathfrak{so}_{2m+1}\C$-representation $V$ is $a_1L^{'}_1+\cdots+a_{m}L^{'}_m$, with each $a_i$  a nonzero half integer, then consider the representation $Sym^2V$, and let $V_1$ be the irreducible $\mathfrak{so}_{2m+1}\C$-submodule of  $Sym^2V$ with highest weight $2a_1L^{'}_1+\cdots+2a_{m}L^{'}_m$. By Weyl's construction, there exists a partition $\lambda=(\lambda_1,\cdots, \lambda_k)$ of a positive integer $d$ with $k\leq m$, such that $V_1\simeq \mathbb{S}_{[\lambda]}W$ as $\mathfrak{so}_{2m+1}\C$-modules. Here $W=\C^{2m+1}$ is the standard representation of $\mathfrak{so}_{2m+1}\C$. Then arguing in the same way as the \textbf{B1} case, we find $\lambda_1\leq 2$. Since as a representation of $\mathfrak{so}_{2m+1}\C$, the highest weight of $ \mathbb{S}_{[\lambda]}W$ is $\lambda_1L^{'}_1+\cdots+\lambda_kL^{'}_k$, we see that $2a_1=\lambda_1\leq 2$. Hence $a_1=\frac{1}{2}$ and the irreducible $\mathfrak{so}_{2m+1}\C$-representation $V$ is the fundamental spin representation. Then the dimension $dim V={2n\choose n}$ must be  a power of $2$. It is elementary to see that this can never happen if $n\geq 2$. So we can exclude this case.

\textbf{Case C}: Suppose $\mathfrak{g}\simeq \mathfrak{sp}_{2m}\C$ is of type C.  By Weyl's construction, there exists a partition $\lambda=(\lambda_1,\cdots, \lambda_k)$ of a positive integer $d$ with $k\leq m$, such that $V\simeq \mathbb{S}_{<\lambda>}W$ as $\mathfrak{sp}_{2m}\C$-modules. Here $W=\C^{2m}$ is the standard representation of $ \mathfrak{sp}_{2m}\C$ (cf. Theorem 17.11 in \cite{FH}). By Lemma \ref{lemma1:representation theory} and arguing in the same way as  the \textbf{B1} case, we find $\lambda_1=1$ and $V=\mathbb{S}_{<(1,\cdots, 1)>}W$ as $\mathfrak{sp}_{2m}\C$-modules. Then it is easy to see the only possibility is $d=1$, $2m={2n\choose n }$ and $\mathfrak{g}=\mathfrak{g}_n$. This gives case (1).

\textbf{Case D}: Suppose $\mathfrak{g}\simeq \mathfrak{so}_{2m}\C$ is of type D. Following Section 18.1 in \cite{FH}, we choose the basis $L^{'}_1, \cdots, L^{'}_{m}$ of the dual space of a Cartan subalgebra (note we add a $'$ to distinguish with the weights of  $\mathfrak{sp}_{2n}\C$). We have three cases to discuss.

\textbf{D1}:  If the highest weight of the irreducible $\mathfrak{so}_{2m}\C$-representation $V$ is $a_1L^{'}_1+\cdots+a_{m-1}L^{'}_{m-1}$, with each $a_i$  an integer, then by Weyl's construction, there exists a partition $\lambda=(\lambda_1,\cdots, \lambda_k)$ of a positive integer $d$ with $k\leq m-1$, such that $V\simeq \mathbb{S}_{[\lambda]}W$ as $\mathfrak{so}_{2m}\C$-modules. Here $W=\C^{2m}$ is the standard representation of $ \mathfrak{so}_{2m}\C$ (cf. Theorem 19.22 in \cite{FH}). By Lemma \ref{lemma1:representation theory} and arguing in the same way as  the \textbf{B1} case, we find $\lambda_1=1$ and $V=\mathbb{S}_{[(1,\cdots, 1)]}W$ as $\mathfrak{so}_{2m}\C$-modules. Then it is easy to see the only possibility is $d=1$, $2m={2n\choose n }$ and $\mathfrak{g}=\mathfrak{g}_n$. This gives case (1).

\textbf{D2}: Suppose the highest weight of the irreducible $\mathfrak{so}_{2m}\C$-representation $V$ is $\lambda_1L^{'}_1+\cdots+\lambda_{m}L^{'}_{m}$, with each $\lambda_i$  an integer and $\lambda_{m}\neq 0$.  Let $\lambda=(\lambda_1,\cdots, \lambda_{m-1}, | \lambda_m |)$ be the partition of $d=\sum_{i=1}^{m-1}\lambda_i+ | \lambda_m |$.The standard representation $W=\C^{2m}$ of $ \mathfrak{so}_{2m}\C$ can also be viewed  as a representation of the complex Lie group $SO_{2m}\C$. By Theorem 19.22 in \cite{FH},  $ \mathbb{S}_{[\lambda]}W$ is an irreducible representation of the complex Lie group $O_{2m}\C$ and  as a representation of $SO_{2m}\C$, we have $\mathbb{S}_{[\lambda]}W=V\oplus V^{'}$, where $V^{'}$ is  conjugate to $V$. In particular, for any $\sigma \in O_{2m}\C$, if $det \sigma=-1$, then $V^{'}=\sigma V$. As a representation of $\mathfrak{sp}_{2n}\C$, consider the weight space decomposition of $W$:
$$
W=\oplus_{\alpha\in\mathfrak{h}^{*}}W_{\alpha}.
$$
Since $W$ is a nontrivial representation of $\mathfrak{sp}_{2n}\C$, we can choose a weight $\alpha_1=a_2L_2+\cdots+a_nL_n$ with $W_{\alpha_1}\neq 0$. Let $W_1=W_{\alpha_1}\oplus W_{-\alpha_1}$, and let $W_2$ be the direct sum of other nonzero weight spaces of $W$. Then $W=W_1\oplus W_2$ and $W_1\perp W_2$ under the standard $\mathfrak{so}_{2m}\C$-invariant symmetric form $Q$, since $Q$ is also invariant under $\mathfrak{sp}_{2n}\C$ and the sum of weights of any two nonzero weight vectors  from $W_1$ and $W_2$ respectively is not zero. Since $Q$ is non-degenerate,  We can choose a basis $e_1,\cdots, e_l$ of $W_{\alpha_1}$ and a basis $e_1^{'}, \cdots, e_l^{'}$ of $W_{-\alpha_1}$, such that
$$
Q(e_i, e_i)=Q(e^{'}_i, e^{'}_i)=0, \forall \ 1\leq i\leq l,
$$
and $\forall \ 1\leq i,j\leq l$,
$$
Q(e_i, e^{'}_j)=\left\{
                  \begin{array}{ll}
                    1, & \hbox{$i=j$;} \\
                    0, & \hbox{$i\neq j$.}
                  \end{array}
                \right.
$$
Define a linear transformation $\sigma$ of $W$ by $\sigma\mid_{W_2}=id$, $\sigma(e_i)=e_i$, $\sigma(e^{'}_i)=e^{'}_i$, $\forall 1\leq i\leq l-1$, and $\sigma$ interchanges $e_l$ and $e^{'}_l$. Then obviously $\sigma \in O_{2m}\C$ and $det \sigma =-1$.
By the proof of Lemma \ref{lemma1:representation theory} and the definition of $\sigma$, we  can find $m$ linearly independent weight  vectors $v_1,\cdots, v_{m}=e_l$ of $\mathfrak{sp}_{2n}\C$ such that
\begin{itemize}
 \item[(1)]$\sigma(v_i)=v_i$, $\forall \ 1\leq i\leq m-1$;
 \item[(2)]  both $v_{\lambda}\cdot c_{\lambda}\in \mathbb{S}_{[\lambda]}W $  and $\sigma(v_{\lambda}\cdot c_{\lambda})\in \mathbb{S}_{[\lambda]}W$  are nonzero;
 \item[(3)] the weight of $v_{\lambda}\cdot c_{\lambda}$  is $b_1L_1+\cdots +b_nL_n$, with $b_1\geq \lambda_1$.
 \end{itemize}
 By the construction,  we see the weight of $\sigma(v_{\lambda}\cdot c_{\lambda})$ is  $b_1^{'}L_1+\cdots+b_n^{'}L_n$ with $b_1^{'}=b_1\geq \lambda_1$. From this, we deduce that in $V$, there always exists a nonzero weight vector of $\mathfrak{sp}_{2n}\C$ with weight $c_1L_1+\cdots+c_nL_n$, $c_1\geq \lambda_1$. Then arguing as before, we get $\lambda_1=1$, hence $\lambda_1=\cdots=\lambda_{m-1}=1$, $\lambda_{m}=\pm 1$. Then it is not difficult to see this can not happen. So this case is excluded.

\textbf{D3}: If the highest weight of the irreducible $\mathfrak{so}_{2m}\C$-representation $V$ is $a_1L^{'}_1+\cdots+a_{m}L^{'}_m$, with each $a_i$  a nonzero half integer, then consider the representation $Sym^2V$, and let $V_1$ be the irreducible $\mathfrak{so}_{2m}\C$-submodule of  $Sym^2V$ with highest weight $2a_1L^{'}_1+\cdots+2a_{m}L^{'}_m$. By the discussion of case \textbf{D2}, if we define the partition $\lambda=(2a_1,\cdots,  2a_{m-1}, 2|a_{m}|)$, then as $\mathfrak{so}_{2m}\C$-modules,   $V_1$ is a direct summand of $\mathbb{S}_{[\lambda]}W$, where  $W=\C^{2m}$ is the standard representation of $ \mathfrak{so}_{2m}\C$, and there exists a nonzero weight vector of $\mathfrak{sp}_{2n}\C$ in $V_1$ with weight $c_1L_1+\cdots+c_nL_n$, $c_1\geq 2a_1$. Since as  a $\mathfrak{sp}_{2n}\C$-module, $V_1$ is a direct summand of $Sym^2 (\wedge^n\C^{2n})$, we find easily that $2a_1\leq 2$. So $a_1=\frac{1}{2}$ and $V$ is a fundamental spin representation of $\mathfrak{so}_{2m}\C$. Then the dimension of $V$ is a power of $2$. On the other hand, $dim V={2n\choose n}$, which can never be a power of $2$ if $n\geq 2$. So we can excludes this case.

\end{proof}

\begin{proposition}\label{prop:PVHS decomposition}
Let  $\V$ be an absolutely irreducible $\C$-PVHS over a quasi-projective variety $S$ with quasi-unipotent local monodromy around each component of $D=\bar S-S$ and $\W$ be a rank $2n$ local system of complex vector spaces over $S$. Suppose $\V\simeq \wedge^n \W$ as local systems. Then $\W$ admits a $\C$-PVHS structure such that the induced $\C$-PVHS on the wedge product $\wedge^n \W$ coincides with the given $\C$-PVHS on $\V$.
\end{proposition}
\begin{proof}
Assume first $D=\emptyset$ to illustrate the idea. By the result of Simpson(cf.\cite{Simpson92}), a complex local system admits a $\C$-PVHS structure if and only if it is fixed by the $\C^*$-action on the moduli space of semisimple representations of $\pi_1$. Now  the wedge $n$ product of $\C^{2n}$ induces a homomorphism $GL(2n)\xrightarrow{\rho_{\wedge^n}}GL({2n\choose n})$, which has a finite kernel. This homomorphism induces the morphism
$$
\phi_{\wedge^n}: \mathfrak{M}(\pi_1(S), GL(2n))^{ss}\to \mathfrak{M}(\pi_1(S),GL({2n\choose n}))^{ss}
$$
between the corresponding moduli spaces of semi-simple representations. By Corollary 9.18 in \cite{Simpson94},  the morphism $\phi_{\wedge^n}$ is finite. Note $\C^{*}$ acts on both moduli spaces continuously via the Hermitian
Yang-Mills metric on the corresponding poly-stable Higgs bundles, and this action is compatible
with $\phi_{\wedge^n}$.
Thus since $[\V=\wedge^n\W]$ is fixed by the $\C^{*}$-action, it follows that $[\W]$ itself is fixed by the $\C^*$-action. Thus $\W$ admits a $\C$-PVHS structure such that it induces a $\C$-PVHS structure on $\wedge^n\W$. By Deligne's uniqueness theorem  of $\C$-PVHS structures on an irreducible local system, it coincides with the given one on $\V$.

Consider the general case. First we show $\W$ has also quasi-unipotent local monodromy. Let $\gamma\in \pi_1(S)$ be a loop around a component of $D=\bar{S}-S$ and $T$ be the corresponding local monodromy of $\W$. Then $\wedge^nT$ is quasi-unipotent. Since  $T=T_sT_u$, where $T_s (T_u)$ is the semisimple (unipotent) part of $T$,  we can assume the eigenvalues of $\wedge^nT$ are all one. Let $\lambda_1,\cdots,\lambda_{2n}$ be the eigenvalues of $T$. Then $\{\lambda_{i_1}\cdots\lambda_{i_n}, 1\leq i_1<\cdots <i_n\leq 2n\}$ are all one. It implies that $\lambda_1=\cdots=\lambda_{2n}=\pm 1$. Thus $T$ is quasi-unipotent. After a finite base change, we assume that the local monodromies are unipotent. By the result of Jost-Zuo \cite{Jost-Zuo}, there exists a harmonic metric on the flat bundle $\W$ with finite energy which makes $\W$ into a Higgs bundle $(F,\eta)$ on $S$. T. Mochizuki \cite{Mochizuki} has further analyzed the singularity of this harmonic metric and in particular shown that $(F,\eta)$ admits a logarithmic extension $(\bar F, \bar \eta)$ with logarithmic poles of Higgs field along $D$. By the uniqueness of such harmonic metrics, the induced metric $\wedge^n\W$ coincides with the Hodge metric given by the $\C$-PVHS $\V$.

Let $C\subset \bar S$  be a general complete intersection curve  of a very ample divisor of $\bar S$. Set $C_0=C-C\cap D$. Taking the restrictions, we obtain $[\W|_{C_0}]\in \mathfrak{M}(\pi_1(C_0), Gl(n))^{ss}$ such that $[\wedge^n\W|_{C_0}]\in \mathfrak{M}(\pi_1(C_0), Gl({2n\choose n}))^{ss}$. By Simpson \cite{Simpson90}, there exists Hermitian-Yang-Mills metrics on polystable Higgs bundles on $C$ with logarithmic poles of Higgs field along $C\cap D$. The $\C^*$-action can be defined on both spaces of semisimple representations on $C_0$ via a Hermitian-Yang-Mills metric on $(\bar F, t\bar \eta), t\in \C^*$. By the same arguments as above, we show that the restriction of $(\bar F,\bar \eta)$ to $C_0$ is a fixed point of the $\C^*$-action. If we choose $C_0$ sufficiently ample, then $(\bar F,\bar \eta)$ is also a fixed point of the $\C^*$-action. Again by Simpson \cite{Simpson92}, $\W$ admits a $\C$-PVHS structure. This concludes the proof.
\end{proof}


Now we consider the monodromy representation associated to $\tilde{\V}_{AR}$:
$$
\rho: \pi_1(\mathfrak{M}_{AR}, s)\rightarrow Aut(V,Q)
$$
Here and in the following part of  this section we keep the notations in Section \ref{section:introduction}.
\begin{theorem}\label{thm:Zariski density}
If $n\geq 3$, then $Mon^0=Aut^0(V,Q)$. That is, the monodromy group of $\tilde{\V}_{AR}$ is Zariski dense in $Aut^0(V,Q)$.
\end{theorem}
\begin{proof}
We first show the monodromy representation is absolutely irreducible.
For otherwise there would exist local systems $\V_1$, $\V_2$ of complex linear spaces over $\mathfrak{M}_{AR}$, such that $\tilde{\V}_{AR}\otimes \C=\V_1\oplus \V_2$. Then by a result of P. Deligne (cf. \cite{D}), there exist $\C$-PVHS structures on $\V_1$ and $\V_2$ such that $\tilde{\V}_{AR}\otimes \C=\V_1\oplus \V_2$ as $\C$-PVHS. So the Higgs bundle $(E, \theta)=(\oplus E^{p,q}, \oplus\theta^{p,q})$ associated to $\tilde{\V}_{AR}\otimes \C$ admits a direct sum decomposition. On the other hand, Proposition \ref{prop:identification of Hodge structures with Jacobian ring} and Proposition \ref{prop:bases} shows that $\forall \ 1\leq q\leq n$, the Higgs map
$$
Sym^q \mathcal{T}_{\mathfrak{M}_{AR}}\xrightarrow{\theta^q}Hom (E^{n,0}, E^{n-q,q})=E^{n-q,q}
$$
is surjective. This is a contradiction with the decomposition of $(E, \theta)$. So we get the monodromy representation is absolutely irreducible.


Consider the  universal family of hyperelliptic curves $g:\mathcal{C}\rightarrow \mathfrak{M}_{hp}$. By Proposition \ref{prop:reduce to V_1}, we have an inclusion $\mathfrak{M}_{hp}\subset \mathfrak{M}_{AR}$ and  $\tilde{\V}_{AR}\mid_{\mathfrak{M}_{hp}}=\wedge^n \V_C$, where $\V_C$ is the weight one $\Q$-PVHS associated to $g$. Suppose the base point $s\in \mathfrak{M}_{hp}$. Denote the  monodromy representation of $g$ by
$$
\tau: \pi_{1}(\mathfrak{M}_{hp},s)\rightarrow Sp(2n, \Q).
$$
Then we have a commutative diagram
\begin{diagram}
 \pi_{1}(\mathfrak{M}_{hp},s) &\rTo^{\tau} &Sp(2n, \Q)\\
\dTo  &    &\dTo_{\rho_{\wedge^n}}\\
\pi_1(\mathfrak{M}_{AR}, s)   &\rTo^{\rho} &Aut(V,Q)
\end{diagram}
where $\rho_{\wedge^n}$ is the homomorphism  induced by the $n$-th wedge product of the standard representation of $Sp(2n, \Q)$.

By Theorem 1 of \cite{A'Campo},  $\tau(\pi_{1}(\mathfrak{M}_{hp},s))$ is Zariski dense in $Sp(2n, \Q)$. So we
get the commutative diagram
\begin{diagram}
Sp(2n, \Q)&  &\rTo^{\rho_{\wedge^n}}& &Aut(V,Q)\\
&\rdInto^{}& & \ruInto^{}\\
& &Mon
\end{diagram}
 Note the complexification
 \begin{displaymath}
Aut^0(V,Q)_{\C}=\left\{
                 \begin{array}{ll}
                  Sp({2n\choose n}, \C), & \hbox{n odd;} \\
                   SO({2n\choose n}, \C), & \hbox{n even.}
                 \end{array}
               \right.
\end{displaymath}
  By a result of Deligne (cf. Corollary 4.2.9 in \cite{D-HodgeII}), $Mon$ is semi-simple. Then apply Proposition \ref{prop:represnetation theory} to the Lie algebra version of the commutative diagram above, we get either
$Mon^0=Aut^0(V,Q)$ or (after a possible finite \'{e}tale base change) there exists a local system $\W$ over $\mathfrak{M}_{AR}$ such that $\tilde{\V}_{AR}\otimes \C=\wedge^n{\W}$. In the latter  case, Proposition \ref{prop:PVHS decomposition} implies that there exists a  $\C$-PVHS structure on $\W$ such that $\tilde{\V}_{AR}\otimes \C=\wedge^n{\W}$ as $\C$-PVHS over $\mathfrak{M}_{AR}$. This would imply $\tilde{\V}_{AR}$ factors through the  $\C$-PVHS $\V_{can}$ over  the type A symmetric domain  $D^{I}_{n,n}$. But this can not happen by Theorem \ref{thm: does not factor canonically}. So we get $Mon^0=Aut^0(V,Q)$.
\end{proof}

\begin{corollary}\label{cor:M-T group is dense}
If $n\geq 3$, then the special Mumford-Tate group of a general member $X$ in $\mathcal{M}_{n,2n+2}$ is $Aut^0(V, Q)$. 
\end{corollary}
\begin{proof}
Consider the good family $\tilde{f}: \tilde{\mathcal{X}}_{AR}\rightarrow \mathfrak{M}_{AR}$ and the associated weight $n$ $\Q$-PVHS $\tilde{\V}_{AR}$. By Deligne and Schoen (see for example Lemma 2.4 \cite{VZ}), the identity  component of the $\Q$-Zariski closure of the monodromy group is a normal subgroup of the special Mumford-Tate group $Hg(\tilde{\V}_{AR})$ of $\tilde{\V}_{AR}$, which is equal to the special Mumford-Tate group of a general closed fiber of $\tilde{f}$. We can easily deduce from  Theorem \ref{thm:Zariski density} that $Mon^0=Hg(\tilde{\V}_{AR})=Aut^0(V, Q)$. Then the corollary follows since the moduli map of $\tilde{f}$ is dominant.
\end{proof}

\begin{corollary}\label{cor:good family Zariski dense}
Let $f: \mathcal{X}\rightarrow S$ be a good family of $\mathcal{M}_{n,2n+2}$ and $\W$ be the associated weight $n$ $\Q$-PVHS.
Let $s\in S$ be a base point and let
$$
\rho: \pi_{1}(S,s)\rightarrow Aut(V, Q)
$$
be the monodromy representation associated to $\W$, where $V=\W_s$ and $Q$ is the bilinear form on $V$ induced by the cup product.  If $n\geq 3$, then the image of $\rho$ is Zariski dense in $Aut^0(V, Q)$.
\end{corollary}
\begin{proof}
By Corollary \ref{cor:M-T group is dense} and the results of Deligne and Schoen we used above, the identity component of the $\Q$-Zariski closure of the monodromy group is a normal subgroup of $Aut^0(V, Q)$. Then the corollary follows from the fact that  $Aut^0(V, Q)$ is an almost simple algebraic group.
\end{proof}




\section{Gross's Geometric realization problem}\label{section:general results}

Motivated by Theorem \ref{thm: does not factor canonically}, we consider a particular subclass of moduli spaces of hyperplane arrangements in projective spaces, namely those related to the CY varieties, and ask whether their (sub) VHSs will realize the canonical PVHSs over type $A$ bounded symmetric domains. More precisely, the question of Gross in this setting is described as follows.
\begin{question}[B. Gross \cite{G}]
Let $D=G_{\R}/K$ be an irreducible type $A$ bounded symmetric domain and $\V_{can}$ be the canonical $\C$-PVHS of CY type over $D$. Does there exist a good family of CY manifolds $f: X\to S$ which is obtained from a crepant resolution of cyclic covers of $\P^n$ branched along $m$ hyperplanes in general position, where $S=D/\Gamma$ with $\Gamma\subset G_{\Q}$ an arithmetic subgroup, such that $\V_{can}$ is the pull back to $D$ of any sub $\C$-PVHS of the $\Q$-PVHS $\V_{X}$ attached to $f$?
\end{question}
Gross stated his question only for the tube domain case. However, it is equally interesting to consider other cases, e.g. the complex ball. Then one has to extend the construction of Gross of $\V_{can}$ (in this case and only in this case it is an $\R$-PVHS) to the remaining cases. This was done in \cite{SZ}. A recent work of Friedmann-Laza \cite{FL} showed that $\C$-PVHS of CY types over bounded symmetric domains come basically from the $\V_{can}$ of Gross and Sheng-Zuo. Note that, in the tube domain case, the affirmative answer to the above question is in fact a weaker reformulation of Dogalchev's conjecture. What we intend to do in this section is to make a definite and complete answer to the geometric realization problem for type $A$ domains with moduli spaces of CY manifolds coming hyperplane arrangements as potential candidate in mind.
In fact, we will prove results analogous to Theorem \ref{thm: does not factor canonically}  and Theorem \ref{thm:Zariski density}. Let us fix the following  notations throughout this section.
\begin{itemize}
\item $m,n,k,r$ are positive integers such that $m=kr$ and $n=m-k-1$.
\item $\zeta=$exp$(\frac{2\pi i}{r})$ is a primitive $r$-root of unit.
\end{itemize}

We call an ordered arrangement $\mathfrak{A}=(H_1,\cdots, H_m)$ of $m$ hyperplanes in $\P^n$ in general position if no $n+1$ of the hyperplanes intersect in a point.
The hyperplanes of the arrangement $\mathfrak{A}$ determine a divisor $H=\sum_{i=1}^{m}H_i$ on $\P^n$. As
the same as the $r=2$ case, this divisor determines a unique $r$-fold cyclic cover $\pi: X\rightarrow \P^n$ that ramifies over $H$ and the canonical line bundle of $X$ is trivial. In the same way as section \ref{subsection:Kummer cover}, we can construct the Kummer cover of $X$ which is smooth projective, so the Hodge structure on $H^{n}(X, \Q)$ is a weight $n$ pure $\Q$-Hodge structure.  In our earlier work \cite{SXZ}, we constructed a crepant resolution $\tilde{X}$ of $X$. Thus the projective variety $\tilde{X}$ is a smooth CY manifold.
\begin{lemma}\label{lemma:resolution not change eigen Hodge structure}
The crepant resolution $\psi: \tilde{X}\rightarrow X$ induces  isomorphisms:
\begin{itemize}
\item[(1)]$H^{p,q}(X, \C)\xrightarrow{\sim} H^{p,q}(\tilde{X},\C)$,\ \ $\forall \ p+q=n$, $p\neq q$.
\item[(2)]$\psi^{*}: H^{n}(X, \C)_{(i)}\xrightarrow{\sim}H^n(\tilde{X},\C)_{(i)}$, \ \ $\forall \ 1\leq i\leq r-1$.
\end{itemize}
\end{lemma}
\begin{proof}
(1) is just Proposition 2.8 in \cite{SXZ}.

(2) can be proved in the same way as Proposition 2.8 in \cite{SXZ}, by replacing everything by its $i$-eigenspace, and noting that by induction, at every blow up step, $\forall \ 1\leq i\leq r-1$, $H^n(E)_{(i)}=0$, where $E$ is the exceptional divisor.
\end{proof}

Let $\mathfrak{M}_{AR}$ denote the coarse moduli space of ordered arrangements of $m$ hyperplanes in $\P^n$ in general position, and let $\mathcal{M}_{n,m}$ denote the coarse moduli space of $\tilde{X}$.
In the same way as the $r=2$ case, $\mathfrak{M}_{AR}$ can be realized as an open subvariety of the affine space $\C^{n(k-1)}$ and it admits a natural family $f:\mathcal{X}_{AR}\rightarrow \mathfrak{M}_{AR}$, where each fiber $f^{-1}(\mathfrak{A})$ is the $r$-fold cyclic cover of $\P^n$ branched along the hyperplane arrangement $\mathfrak{A}$. It is easy to see the crepant resolution  in  \cite{SXZ} gives a simultaneous crepant resolution $\pi: \tilde{\mathcal{X}}_{AR}\rightarrow \mathcal{X}_{AR}$ for the family $f$. We denote this smooth projective family of CY manifolds by $\tilde{f}:\tilde{\mathcal{X}}_{AR}\rightarrow \mathfrak{M}_{AR}$.

Let $\mathfrak{M}_{C}$ be the moduli space of ordered distinct $m$ points on $\P^1$ and  $g:\mathcal{C}\rightarrow \mathfrak{M}_{C}$ be the  universal family of $r$-fold cyclic covers of $\P^1$ branched at $m$ distinct points.

 We consider the $\Q-$VHS attached to the three families $f$, $\tilde{f}$, $g$:
 $$
 \V:=R^nf_{*}\Q, \ \ \tilde{\V}:=(R^n\tilde{f}_{*}\Q)_{pr}, \ \ \V_C:=R^1g_{*}\Q.
 $$
 Since $\Z/r\Z$ acts naturally on the three families, we have a decomposition of the three $\Q$-VHS into eigen-sub $\C$-VHS:
 $$
 \V\otimes\C=\oplus_{i=0}^{r-1}\V_{(i)}, \ \  \tilde{\V}\otimes\C=\oplus_{i=0}^{r-1}\tilde{\V}_{(i)},
 \ \ \V_C\otimes\C=\oplus_{i=0}^{r-1}\V_{C(i)}
 $$
 \begin{proposition}\label{prop:general results, summary1}
  $\forall \ 1\leq i\leq r-1$, we have:
  \begin{itemize}
  \item[(1)] the crepant resolution induces an isomorphism of $\C$-PVHS: $\tilde{\V}_{(i)}\simeq \V_{(i)}$;
  \item[(2)] there is an embedding $\mathfrak{M}_C\hookrightarrow \mathfrak{M}_{AR}$ such that $\V_{(i)}\mid_{\mathfrak{M}_{C}}\simeq \wedge^n\V_{C(i)}$;
  \item[(3)] as a $\C$-PVHS of weight $n$, the Hodge numbers of $\V_{(1)}$ are:
  $$
  h^{n-q,q}= \left\{
               \begin{array}{ll}
                 {n\choose q}{k-1\choose q}, & \hbox{$\ 0\leq q\leq k-1$;} \\
                 0, & \hbox{$k\leq q\leq n$.}
               \end{array}
             \right.
$$
  \end{itemize}
\end{proposition}
\begin{proof}
(1) follows from Lemma \ref{lemma:resolution not change eigen Hodge structure}.

The proof of (2) and (3) is similar to that of Proposition \ref{prop:reduce to V_1}.
\end{proof}

Analogous to the $r=2$ case, for any $2\leq k\leq n+1$, on the  type $A$ symmetric domain $D^{I}_{n,k-1}=\frac{SU(n,k-1)}{S(U(n)\times U(k-1))}$, there is a canonical $\C$-PVHS $\V_{can}$, which has the same Hodge numbers as $\V_{(1)}$. For details of the construction, one can see \cite{SZ}. From the construction, one can deduce in the same way as Proposition \ref{propo:Hodge numbers and chara. variety of canonical VHS} that \begin{proposition}\label{prop: charac.variety of canonical VHS in general case}
The first characteristic variety of $\V_{can}$ is $C_{1,s}\simeq \P^{n-1}\times \P^{k-2}, \ \ \forall s\in D^{I}_{n,k-1}$.
\end{proposition}
In order to calculate the characteristic varieties of $\V$, we use the same method as the $r=2$ case and reduce the situation to a calculation in a Jacobian ring. The construction and properties of Jacobian rings are similar to the $r=2$ case. So we only summarize and state the results we need.

For each parameter $a\in \C^{n(k-1)}$, there is a bi-graded $\C$-algebra $R=\oplus_{p,q\geq 0} R_{(p,q)}$ and the group $N=\oplus_{j=0}^{m-1}\Z/r\Z$ acts on $R$, preserving the grading. Consider the summation homomorphism  $\oplus_{j=0}^{m-1}\Z/r\Z\xrightarrow{\sum}\Z/r\Z$ and let $N_1$ be the kernel of this homomorphism. Define $R^{N_1}=\oplus_{p,q\geq 0}R_{(p,q)}^{N_1}$ to be the $N_1$-invariant part of $R$, then the cyclic group $\Z/r\Z=N/N_1=<\sigma>$ acts on $R^{N_1}$. Let $R_{(p,q)(i)}^{N_1}=\{\alpha\in R_{(p,q)}^{N_1}\mid \sigma (\alpha) =\zeta^i \alpha\}$ be the $i$-th eigen space of $R_{(p,q)}^{N_1}$ under the action of $\Z/r\Z$. Then analogous to Proposition \ref{prop:identification of Hodge structures with Jacobian ring}, we have
\begin{proposition}\label{prop:identification of Hodge structures with Jacobian ring in general case}
\begin{itemize}
\item[(1)] $\forall \ 1\leq i\leq r-1$, $\forall \ 0\leq q\leq n$, $H^{n-q,q}(X)_{(i)}\simeq R^{N_1}_{(q,qr)(i-1)}$, where $X$ is the cyclic cover of $\P^n$ corresponding to the parameter $a\in \mathfrak{M}_{AR}\subset \C^{n(k-1)}$. In particular, $H^{n-q,q}(X)_{(1)}\simeq R^{N_1}_{(q,qr)(0)}=R^{N}_{(q,qr)}$;
\item[(2)] $\forall \ 1\leq i\leq r-1$, $\forall \ 0\leq q\leq n$,  we have a commutative diagram
\begin{equation}\notag
\begin{array}{ccc}
  T_{\mathfrak{M}_{AR},a}\otimes H^{n-q,q}(X)_{(i)} & \xrightarrow{\theta^{n-q,q}} & H^{n-q-1,q+1}(X)_{(i)} \\
   \downarrow{\simeq}&  & \downarrow{\simeq} \\
  R^{N}_{(1,r)}\otimes R^{N_1}_{(q,qr)(i-1)} &\xrightarrow{} & R^{N_1}_{(q+1,qr+r)(i-1)}.
\end{array}
\end{equation}
\end{itemize}
Here  the lower horizontal arrow is the ring  multiplication map.
\end{proposition}
For a parameter $a\in \C^{n(k-1)}$, we define a subvariety in the projective space $\P(R^{N}_{(1,r)})$ as follows
$$
 C^{'}_{1,a}:=\{[\alpha]\in \P(R^{N}_{(1,r)})\mid \alpha^2=0 \in R^{N}_{(2,2r)}  \}.
$$
We have the following upper bound of the dimension of $C^{'}_{1,a}$.
\begin{proposition}\label{prop:Jacobian ring: dim upper bound in general case}
If $n\geq 3$, $k\geq 3$, then for generic $a\in \C^{n(k-1)}$,  dim $C^{'}_{1,a} \leq 2$.
\end{proposition}
For the proof of this proposition, one can follow the proof of Proposition \ref{prop:Jacobian ring: dim upper bound} without any difficulty.

The following theorem is a generalization of Theorem \ref{thm: does not factor canonically}.
\begin{theorem}\label{thm:does not factor in general case}
If $n\geq 3$, $k\geq 3$,   then
\begin{itemize}
\item[(1)] $\tilde{\mathcal{X}}_{AR}\xrightarrow{\tilde{f}}\mathfrak{M}_{AR}$ is a good family for the coarse moduli space  $\mathcal{M}_{n,m}$;
\item[(2)] Let $f: \mathcal{X}\rightarrow S$ be a good family of $\mathcal{M}_{n,m}$ and $\W=(R^nf_{*}\Q)_{pr}$ be the associated weight $n$ $\Q$-PVHS. Then any sub $\C$-VHS $\tilde{\W}$ of Calabi-Yau type in $\W\otimes \C$   does not factor through the  $\C$-PVHS $\V_{can}$ over the type $A$ symmetric domain $D_{n,k-1}^I$. In particular, the first eigenspace $\tilde{\V}_{(1)}$ of $\tilde{\V}$ associated to the family $\tilde{f}$ in (1) does not factor through $\V_{can}$ over $D_{n,k-1}^I$.
    \end{itemize}
\end{theorem}
\begin{proof}
(1) For each fiber $\tilde{X}$ of the family $\tilde{f}$, by Lemma \ref{lemma:resolution not change eigen Hodge structure} and Proposition \ref{prop:identification of Hodge structures with Jacobian ring in general case}, we can identify $H^{n-1,1}(\tilde{X} ,\C)$ with the Jacobian ring $R^{N}_{(1,r)}$ and the Higgs map
$$
T_{\mathfrak{M}_{AR},s}\otimes H^{n,0}(\tilde{X},\C)\xrightarrow{\theta^{n,0}}H^{n-1,1}(\tilde{X},\C)
$$
can be identified with the multiplication map
$$
R^{N}_{(1,r)}\otimes R^{N}_{(0,0)}\rightarrow R^{N}_{(1,r)}
$$
which is obviously an isomorphism. By the local Torelli theorem for  CY manifolds, we get the Kodaira-Spencer map of $\tilde{f}$ is an isomorphism at each point $s\in \mathfrak{M}_{AR}$. This shows $\tilde{f}$ is a good family.

(2) It suffices to prove the statement for the good family $\tilde{\mathcal{X}}_{AR}\xrightarrow{\tilde{f}}\mathfrak{M}_{AR}$.  Recall that for this good family, there is a decomposition of the associated $\C$-VHS: $\tilde{\V}\otimes\C=\oplus_{i=0}^{r-1}\tilde{\V}_{(i)}$. Suppose  we have a sub $\C$-VHS $\tilde{\W}$ of $\tilde{\V}$, a  nonempty (analytically) open subset $U\subset S$ and a holomorphic map $j: U\rightarrow D_{n,k-1}^I$, such that $\tilde{\W}$ is of Calabi-Yau type, and $\tilde{\W}\mid_{U}\simeq j^{*}\V_{can}$ as $\C$-VHS. Since both $\tilde{\W}$ and $\V_{can}$ are of Calabi-Yau type, $j$ is a local isomorphism, so we can assume $j$ is an open embedding of complex manifolds.  Let  $(\tilde{E}, \tilde{\theta})$ and $(E_{can}, \theta_{can})$ be the Higgs bundles corresponding to $\tilde{\W}$ and $\V_{can}$ respectively, then under the embedding $j$, $(\tilde{E}, \tilde{\theta})\mid_{U}=(E_{can}, \theta_{can})\mid_{U}$. It can be seen easily from the description of $(E_{can}, \theta_{can})$ (cf. Section 4.1 in \cite{SZ}) that $\forall \ s\in D_{n,k-1}^I$, $\forall \ q\geq 0$, the map $\theta^{q}: Sym^q T_{D_{n,k-1}^I, s}\otimes E^{n,0}_{can}\rightarrow E_{can}^{n-q,q}$ is surjective. Moreover, by Proposition \ref{prop: charac.variety of canonical VHS in general case}, the first characteristic variety of $\V_{can}$ at each point of $D_{n,k-1}^I$ is isomorphic to $\P^{n-1}\times \P^{k-2}$.  So  the Higgs bundle $\tilde{E}$ also satisfies the following two properties:
\begin{itemize}
\item[(1)] $\forall \ s\in U$, $\forall \ q\geq 0$, the map $\theta^{q}: Sym^q T_{U, s}\otimes \tilde{E}^{n,0}\rightarrow \tilde{E}^{n-q,q}$ is surjective.
\item[(2)] $\forall \ s \in U$, the first characteristic variety of $\tilde{\W}$ at $s$ isomorphic to  $\P^{n-1}\times \P^{k-2}$.
\end{itemize}
Since the $\Z/r\Z$-invariant part $\tilde{\V}_{(0)}$ of $\tilde{\V}$ is obviously a constant $\C$-VHS, $(1)$ implies that we must have  $\tilde{\W}\subset \oplus_{i=1}^{r-1}\tilde{\V}_{(i)}$. Then by Proposition \ref{prop:identification of Hodge structures with Jacobian ring in general case} and taking into account the Hodge numbers of $\V_{can}$, we can translate the properties of $\tilde{E}$  above to the following properties of the  Jacobian ring $R$: there exists  a nonempty open subset $U$ of the parameter space  $\C^{n(k-1)}$, and  for any parameter $ a \in U$, there is an element $\beta\in R^{N_1}$, such that:
\begin{itemize}
\item[$(1)^{'}$] $\forall \ 0\leq q\leq k-1$, the dimension of the linear space $J_q:=\{\beta \cdot \gamma\in R^{N_1}\mid \gamma\in Sym^q R^{N}_{(1,r)}\}$ is ${n\choose q}{k-1\choose q}$.
\item[$(2)^{'}$] the variety $\tilde{C}_{1,a}:=\{\alpha\in \P(R^{N}_{(1,r)})\mid \beta \cdot \alpha^2=0\}$ is isomorphic to $\P^{n-1}\times \P^{k-2}$.
\end{itemize}
It is easy to see from the definition of Jacobian ring that
$$
J_q=\{\beta \cdot \gamma\in R^{N_1}\mid \gamma\in Sym^q R^{N}_{(1,r)}\}=\{\beta \cdot \gamma\in R^{N_1}\mid \gamma\in  R^{N}_{(q,qr)}\}
$$
Note also by Proposition \ref{prop:identification of Hodge structures with Jacobian ring in general case}, the dimension of the linear space $R^{N}_{(q,qr)}$ is  ${n\choose q}{k-1\choose q}$, equal to that of $J_q$.  So  we get the map of multiplication by $\beta$:
$$
R^{N}_{(q,qr)}\xrightarrow{\cdot \beta}J_q
$$
is a linear  isomorphism.
From this by taking $q=2$ we deduce that
$$
\tilde{C}_{1,a}=\{\alpha\in \P(R^{N}_{(1,r)})\mid \beta \cdot \alpha^2=0\}=\{\alpha\in \P(R^{N}_{(1,r)})\mid \alpha^2=0\}= C^{'}_{1,a}.
$$
Then $(2)^{'}$ implies the dimension of $C^{'}_{1,a}$ is $n+k-3\geq 3$, and this contradicts with Proposition \ref{prop:Jacobian ring: dim upper bound in general case}. So we finally get that any sub $\C$-VHS $\W$  of Calabi-Yau type in $\V$  does not factor through $\V_{can}$.
\end{proof}

 The above theorem leaves tiny possibilities for a positive answer of Gross's question, since for simple reasons, it is easy to exclude any other type $A$ domain but $D^{I}_{n,m-n-2}$ for $\mathcal{M}_{n,m}$. Indeed, we can list each of the remaining cases as follows:\\
$\mathcal{M}_{1,4}$: this is the starting point.\\
$\mathcal{M}_{2,6}$: This is a four dimensional family of K3 surfaces with generic Picard number 16. A detailed study of this family was done in Matsumoto-Sasaki-Yoshida \cite{MSY}. The connection of the weight two Hodge structure of such a K3 surface with the weight one Hodge structure of an abelian variety in view of Kuga-Satake construction was geometrically realized by K. Paranjape \cite{Pa}.\\
$\mathcal{M}_{3,6},\mathcal{M}_{5,8},\mathcal{M}_{9,12}$: These are only cases with $n\geq 3$ and $m-n=3$ which can realize the problem of Gross (the partial compactification issue however remains to be done). Our earlier work \cite{SXZ} studied the series $\mathcal{M}_{n,n+3}, n\geq 3$ and related the Hodge structure with that in Deligne-Mostow \cite{DM}. Besides loc. cit., the classification was obtained thanks to the works of Mostow \cite{Mostow0}-\cite{Mostow} on discrete subgroups of the automorphism group of a complex ball.

Now we study the monodromy representation of the family $\tilde{\mathcal{X}}_{AR}\xrightarrow{\tilde{f}}\mathfrak{M}_{AR}$. Suppose $r\geq 3$, then there is a weight $n$ $\R$-VHS $\tilde{\V}_{\R, (1)}$ such that
$$
\tilde{\V}_{\R, (1)}\otimes \C=\tilde{\V}_{(1)}\oplus \tilde{\V}_{(r-1)}
$$
and the polarization on $\tilde{\V}$ induces  a parallel hermitian form $h$ on $\tilde{\V}_{(1)}$ of signature $(p,q)$. Here by Proposition \ref{prop:general results, summary1}
$$
p=\sum_{i=0}^{[\frac{k-1}{2}]}{n\choose 2i}{k-1\choose 2i},  \ q=\sum_{i=0}^{[\frac{k}{2}]-1}{n\choose 2i+1}{k-1\choose 2i+1}.
$$
Take a base point $s\in \mathfrak{M}_{AR}$ and consider the real monodromy representation
$$
\rho: \pi_1(\mathfrak{M}_{AR}, s)\rightarrow GL(V)
$$
where $V$ is the fiber of $\tilde{\V}_{\R, (1)}$ over $s$. Let $Mon_{\R}$ be the Zariski closure of $\rho(\pi_1(\mathfrak{M}_{AR}, s))$ in the real algebraic group $GL(V)$.
The following theorem is parallel to Theorem \ref{thm:Zariski density}.
\begin{theorem}\label{thm:Zariski density in general case}
Suppose $r\geq 3$ and $k\geq 3$, then the identity component $Mon^0_{\R}$ of $Mon_{\R}$ is isomorphic to $SU(p,q)$.
\end{theorem}

Before starting the proof of this theorem, we  state two propositions parallel to Proposition \ref{prop:represnetation theory} and Proposition  \ref{prop:PVHS decomposition}.
\begin{proposition}\label{prop:represnetation theory in general case}
The $n$-th wedge product $V=\wedge^n \C^{n+k-1}$ of the standard representation of $\mathfrak{sl}_{n+k-1}\C$ induces an embedding $\mathfrak{sl}_{n+k-1}\C\hookrightarrow \mathfrak{sl}_{p+q}$. Suppose $k\geq 3$ and  $\mathfrak{g}$ is a complex semi-simple Lie algebra lying  between  $\mathfrak{sl}_{n+k-1}\C$ and  $\mathfrak{sl}_{p+q}$ such that the induced representation of $\mathfrak{g}$ on $V$ is irreducible, then $\mathfrak{g}$ is one of the following:
\begin{itemize}
\item[(1)] $\mathfrak{sl}_{p+q}$,
\item[(2)] $sl_{n+k-1}\C$, in which case the induced representation of $\mathfrak{g}$ on $V$ is isomorphic to the n-th wedge product of the standard  representation on $\C^{n+k-1}$.
\end{itemize}
\end{proposition}
\begin{proof}
The proof of Proposition \ref{prop:represnetation theory} goes through without difficulty.
\end{proof}

\begin{proposition}\label{prop:PVHS decomposition in general case}
Let  $\V$ be a $\C$-PVHS over a quasi-projective variety $S$ with quasi-unipotent local monodromy around each component of $D=\bar S-S$ and $\W$ be a rank $n+k-1$ local system of complex vector spaces over $S$. Suppose $\V\simeq \wedge^n \W$ as local systems and $k\geq 3$. Then $\W$ admits a $\C$-PVHS structure such that the induced $\C$-PVHS on the wedge product $\wedge^n \W$ coincides with the given $\C$-PVHS on $\V$.
\end{proposition}
\begin{proof}
The proof of Proposition \ref{prop:PVHS decomposition} goes through without difficulty.
\end{proof}

Now we can proceed to the proof of Theorem \ref{thm:Zariski density in general case}.

\begin{proof}
Consider the family of curves $g: \mathcal{C}\rightarrow \mathfrak{M}_C$ and assume the  base point $s\in \mathfrak{M}_{C}\hookrightarrow \mathfrak{M}_{AR}$. Let  $C$ and $X$ be the fibers over $s$  of the families  $\mathcal{C}\xrightarrow{g} \mathfrak{M}_C$ and $\mathcal{X}_{AR}\xrightarrow{f}\mathfrak{M}_{AR}$ respectively. By Proposition \ref{prop:general results, summary1},
$$
H^n(X,\C)_{(1)}=\wedge^nH^1(C,\C)_{(1)}.
$$
The embedding $\R\hookrightarrow \C $ allows to consider $H^1(C, \C)_{(1)}$ and $H^n(X, \C)_{(1)}$ as $\R$-vector spaces.
Consider the monodromy representation of the family $\mathcal{X}_{AR}\xrightarrow{f}\mathfrak{M}_{AR}$:
$$
\tau: \pi_1(\mathfrak{M}_{AR}, s)\rightarrow GL_{\R}(H^n(X, \C)_{(1)}).
$$
Since $\tilde{\V}_{(1)}=\V_{(1)}$, and $\tilde{\V}_{\R,(1)}\simeq \tilde{\V}_{(1)}$ as $\R$-local systems, we can identify $Mon_{\R}$ with the Zariski closure of $\tau(\pi_1(\mathfrak{M}_{AR}, s))$ in $GL_{\R}(H^n(X, \C)_{(1)})$. By Theorem 5.1.1 in \cite{Rhode},  for the family $g: \mathcal{C}\rightarrow \mathfrak{M}_C$, the identity component of the Zariski closure of the monodromy representation
$$
\pi_1(\mathfrak{M}_{C}, s)\rightarrow GL_{\R}(H^1(C, \C)_{(1)})
$$
is $SU(n,k-1)$. So similar to the proof of Theorem \ref{thm:Zariski density} we have a commutative diagram
\begin{diagram}
SU(n,k-1)&  &\rTo^{\rho_{\wedge^n}}& &Aut(H^n(X, \C)_{(1)},h)=U(p,q)\\
&\rdInto^{}& & \ruInto^{}\\
& &Mon_{\R}^0
\end{diagram}
where the homomorphism  $\rho_{\wedge^n}$ is induced by the $n$-th wedge product of the standard representation of $SU(n,k-1)$. Note that there exists a parallel $\Z[\zeta]$-lattice $H^n(X, \Z[\zeta])_{(1)}$ inside $H^n(X, \C)_{(1)}$. From this one can deduce that $Mon^0_{\R}\subset SU(p,q)$.  Arguing in the same way as the proof of Theorem \ref{thm:Zariski density}, we can show the complex representation of $Mon_{\R}^0$ on $H^n(X, \C)_{(1)}$ is irreducible. Since $Mon_{\R}$ is semi-simple by a result of Deligne(cf. Corollary 4.2.9 in \cite{D-HodgeII}),   by taking the complexification  and using Proposition \ref{prop:represnetation theory in general case}, Proposition \ref{prop:PVHS decomposition in general case} in the same way as Theorem \ref{thm:Zariski density}, we get $Mon_{\R}^0=SU(p,q)$.
\end{proof}


\section{Two calculations}\label{sec:two calculations}
In this section, we do two  concrete calculations, one is for the primitive Hodge numbers of the CY manifold $\tilde{X}$, and the other is for the dimension of the characteristic variety $C_{1,a}$.
\subsection{Calculation of Hodge numbers}\label{subsec:Hodge number calculation}
Let us  keep the notations in Section \ref{subsection:double cover and crepant resolution}. We will calculate the primitive Hodge numbers of $\tilde{X}$ and complete the proof of Proposition \ref{hodge number}. Let $\sigma: \tilde\P^n\to \P^n$ be the composite of all blow-ups and $\tilde H$ the strict transform of $H$. Let $\tilde \sL$ be a line bundle on $\tilde X$ such that $\tilde \sL^2=\sO_{\tilde X}(\tilde H)$. As
$$
\tilde \pi_*\Omega^p_{\tilde X}=\Omega^p_{\tilde \P^n}\oplus \Omega^p_{\tilde \P^n}(\log \tilde H)\otimes \tilde \sL^{-1},
$$
it follows that
$$
H^{p,q}_{prim}(\tilde X)=H^q(\tilde \P^n, \Omega^p_{\tilde \P^n}(\log \tilde H)\otimes \tilde \sL^{-1}).
$$
\begin{claim}
For all $k\neq q$, $H^k(\Omega^p_{\tilde \P^n}(\log \tilde H)\otimes \tilde \sL^{-1})=0$. Therefore,
$$
\chi(\Omega^p_{\tilde \P^n}(\log \tilde H)\otimes \tilde{\mathcal{L}}^{-1})=(-1)^q\dim H^q(\Omega^p_{\tilde \P^n}(\log \tilde H)\otimes \tilde \sL^{-1}).
$$
\end{claim}
\begin{proof}
This is a direct application of the vanishing result {\cite[Proposition 6.1]{EV}}.
\end{proof}
Next, we show that the Euler characteristic keep unchanged under resolution. Namely, we have the following
\begin{claim}\label{invariance}
Put $\sL=\sO_{\P^n}(n+1)$. It holds that
$$
\chi(\P^n, \Omega^p_{\P^n}(\log H)\otimes \sL^{-1})=\chi(\tilde \P^n, \Omega^p_{\tilde \P^n}(\log \tilde H)\otimes \tilde \sL^{-1}).
$$
\end{claim}
The claim follows from a general consideration, which we postpone after stating the last
\begin{claim}\label{formula}
One has the following formula:
$$
\chi(\P^n,\Omega^{p}_{\P^n}(\log H)\otimes \sL^{-1})=(-1)^q{n \choose p}^2.
$$
\end{claim}
It is clear that the above claims implies Proposition \ref{hodge number}. \\

{\itshape Proof of Claim \ref{invariance}:}\\

Let $X$ be an $n$-dimensional smooth projective variety, $D_X \subset X$ a SNCD(simple normal crossing divisor). Let $Z \subset D_X$ be a smooth irreducible component of the singularities of $D_X$.  Let
$\sigma: Y\stackrel{\sigma}{\longrightarrow} X$ be the blow-up of $X$ along  $Z$ and $E$ be the exceptional divisor. Put $D_Y:=\sigma^{*}D_X-2E$.

\begin{proposition}
Notation as above. Let $\sL_{X}$ be an ample invertible sheaf on $X$. Put $\sL_Y:=\sigma^{*}\sL_X-E$. Then it holds that
\begin{eqnarray*}
\chi(X, \Omega^{p}_{X}(\log D_X)\otimes \sL_X^{-1}) &=&
\chi(Y, \Omega^{p}_Y(\log D_Y)\otimes \sL_Y^{-1}).
\end{eqnarray*}
\end{proposition}
\begin{proof}
For $p\geq 1$,the residue exact sequence reads (cf. {\cite[Properties 2.3]{EV}}):
$$
   0\rightarrow \Omega^{p}_Y(\log D_Y)\ \rightarrow \
   \sigma^{*}\Omega^{p}_{X}(\log D_X)=\Omega^p_X(\log \sigma^*D_X)\ \stackrel{\mathrm{res}}{\longrightarrow} \ \Omega^{p-1}_{E}(\log D_Y\cdot E)\ \to \ 0.
$$
We shall tensor the above short exact sequence with $\sigma^{*}\sL_X^{-1}\otimes
\sO_{Y}(E)$ and take the Euler characteristic of the resulting exact sequence. First, applying the Leray spectral squence to the morphism $\sigma$, it follows that
$$
\chi(\sigma^*\Omega^{p}_{X}(\log D_X)\otimes \sigma^{*}\sL_X^{-1}\otimes
\sO_{Y}(E))=\chi(\Omega^{p}_{X}(\log D_X)\otimes \sL_X^{-1}).
 $$
Furthermore, as $\sL_X$ is ample, one has
\begin{eqnarray*}
  \Omega^{p-1}_{E}(\log D_Y\cdot E)\otimes\sigma^{*}\sL_{X}^{-1}\otimes\sO_E(E)&=&\Omega^{p-1}_{E}(\log D_Y\cdot E)\otimes\sO_E(-1).
\end{eqnarray*}
Thus we get
$$
\chi(\Omega^{p}_{X}(\log D_X)\otimes \sL_X^{-1}) = \chi(\Omega^{p}_Y(\log D_Y)\otimes
\sL_Y^{-1})+\chi(\Omega^{p-1}_{E}(\log D_Y\cdot E)(-1)).
$$
It is to show $\chi(\Omega^{p}_{E}(\log D_Y\cdot E)(-1))=0$ for all $p\geq 0$. We proceed by induction on $p$. Write $D=D_Y\cdot E=S_1+S_2+F_1+\cdots+F_k$, where
$S_i,\ i=1,2$ are two sections of $\sigma: E\to Z$ and
$F_i\stackrel{\sigma}{\to}H_i\subset Z$ are $\P^1$-bundles over
hypersurfaces in $Z$ for $i=1,\cdots,k$. Then the residue sequence
reads:
\begin{displaymath}
\begin{array}{lllllllllllll}
   0&\to&\Omega_E&\to&\Omega_{E}(\log D)&\stackrel{\mathrm{res}}{\to}&
 \bigoplus_{i=1}^{2}\sO_{S_i}\oplus\bigoplus_{j=1}^{k}\sO_{F_j}&\to&0.
\end{array}
\end{displaymath}
Write $\sN:=\sN_{Z/X}=\sO_{Z}(D_1)\oplus\sO_{Z}(D_2)$. Then
$E=\mathbf{Proj}(\sN^{*})$.  For $i=1,2$, we assume the quotient
invertible sheaf $\sO_{Z}(-D_i)$ of $\sN^{*}$ gives the section
map $s_i: Z\to E$ and the image of $s_i$ is just $S_i$. Then
\begin{eqnarray*}
   \sO_{S_i}\otimes \sO_{E}(-1)&=& s_{i*}(\sO_{Z}\otimes s_{i*}\sO_{E}(-1)) \\
    &=& s_{i*}(\sO_{Z}\otimes \sO_{Z}(D_i))   \\
    &=& s_{i*}(\sO_{Z}(D_i)).
\end{eqnarray*}
Thus
\begin{eqnarray*}
\chi(E,\bigoplus^{2}_{i=1} \sO_{S_i}\otimes
\sO_{E}(-1))&=&\sum^{2}_{i=1}\chi(E, s_{i*}(\sO_{Z}(D_i)))\\
&=&\sum^{2}_{i=1}\chi(Z, \sO_{Z}(D_i))\\
&=&\chi(Z,\sN)
\end{eqnarray*}
And obviously, $\sO_{F_i}\otimes\sO_{E}(-1)=\sO_{F_i}(-1)$. Now we
shall compute $\chi(\Omega_E\otimes \sO_{E}(-1))$. Recall the
following exact sequence:
\begin{displaymath}
\begin{array}{lllllllllllll}
0&\longrightarrow & \sigma^{*}\Omega_Z&\longrightarrow
&\Omega_E&\longrightarrow&\Omega_{E/Z}&\longrightarrow&0
\end{array}
\end{displaymath}
Tensoring with $\sO_E(-1)$ and taking the Euler characteristic, we obtain
\begin{eqnarray*}
\chi(\Omega_E \otimes\sO_{E}(-1))&=& \chi(\sigma^{*}\Omega_Z \otimes\sO_{E}(-1))+\chi(\Omega_{E/Z}\otimes\sO_{E}(-1)).\\
\end{eqnarray*}
For $\mathrm{R}^i\sigma_{*}(\sO_{E}(-1))=0\ \textrm{for all}\
i$, $\chi(\sigma^{*}\Omega_Z \otimes\sO_{E}(-1))=0$. And since
\begin{displaymath}
 \mathrm{R}^i\sigma_{*}(\Omega_{E/Z}\otimes\sO_{E}(-1))=\left\{\begin{array}{ll}
0& \textrm{i=0}\\
 \sN_{Z/X}& \textrm{i=1}\end{array}
\right.
\end{displaymath}
one computes that
\begin{eqnarray*}
\chi( \Omega_{E/Z}\otimes
\sO_{E}(-1))&=&\sum_{i,j}(-1)^{i+j}\dim H^i(Z,\mathrm{R}^j\sigma_{*}\Omega_{E/Z}\otimes
\sO_{E}(-1)) \\
  &=&-\sum_{i}(-1)^{i}\dim H^i(Z,\sN) \\
  &=&-\chi({Z,\sN}).
\end{eqnarray*}
Therefore, it follows that
\begin{eqnarray*}
  \chi(\Omega_{E}(\log D)\otimes \sO_{E}(-1)) &=& \chi(\Omega_E\otimes \sO_{E}(-1))+\sum_{i=1}^{2}\chi(\sO_{S_i}\otimes
\sO_{E}(-1)) \\
   &+& \sum_{j=1}^{k}\chi(\sO_{F_i}\otimes\sO_{E}(-1)) \\
    &=& -\chi({Z,\sN})+\chi({Z,\sN}) \\
    &=& 0
\end{eqnarray*}
This shows the $p=1$ case.  Consider the residue sequence
along $F_k$,
\begin{displaymath}
\begin{array}{lllllllllllll}
0&\to &\Omega^{p}_{E}(\log D-F_k)&\to &\Omega^{p}_{E}(\log
D)&\to&\Omega^{p-1}_{F_k}(\log (D-F_k)\cdot F_k)&\to&0.
\end{array}
\end{displaymath}
By the induction hypothesis, we have that
\begin{eqnarray*}
  \chi(\Omega^{p}_{E}(\log D)\otimes\sO_{E}(-1)) &=& \chi(\Omega^{p}_{E}(\log D-F_k)\otimes\sO_{E}(-1)),
\end{eqnarray*}
and then
\begin{eqnarray*}
  \chi(\Omega^{p}_{E}(\log D)\otimes\sO_{E}(-1)) &=& \chi(\Omega^{p}_{E}(S_1+S_2)\otimes\sO_{E}(-1))
\end{eqnarray*}
Consider the following three short exact sequences:
\begin{displaymath}
\begin{array}{lllllllllllll}
0&\longrightarrow &\Omega^{p}_{E}(\log S_1)&\longrightarrow
&\Omega^{p}_{E}(\log
S_1+S_2)&\longrightarrow&\Omega^{p-1}_{S_2}&\longrightarrow&0,
\end{array}
\end{displaymath}

\begin{displaymath}
\begin{array}{lllllllllllll}
0&\longrightarrow &\Omega^{p}_{E}&\longrightarrow
&\Omega^{p}_{E}(\log
S_1)&\longrightarrow&\Omega^{p-1}_{S_1}&\longrightarrow&0,
\end{array}
\end{displaymath}

\begin{displaymath}
\begin{array}{lllllllllllll}
0&\longrightarrow & \sigma^{*}\Omega_Z^p&\longrightarrow
&\Omega_E^p&\longrightarrow&\sigma^{*}\Omega_Z^{p-1}\otimes\Omega_{E/Z}&\longrightarrow&0.
\end{array}
\end{displaymath}
From the last sequence, it follows that
\begin{eqnarray*}
\chi( \Omega_{E}^p\otimes
\sO_{E}(-1))&=&\chi(\sigma^{*}\Omega_Z^{p-1}\otimes\Omega_{E/Z}\otimes
\sO_{E}(-1))\\
&=&\sum_{i,j}(-1)^{i+j}\dim H^i(Z,\Omega_Z^{p-1}\otimes\mathrm{R}^j\sigma_{*}(\Omega_{E/Z}\otimes
\sO_{E}(-1))) \\
  &=&-\sum_{i}(-1)^{i}\dim H^i(Z,\Omega_Z^{p-1}\otimes\sN) \\
  &=&-\chi({Z,\Omega_Z^{p-1}\otimes\sN}).
\end{eqnarray*}
Moreover, we compute that
\begin{eqnarray*}
\sum^{2}_{i=1}\chi(S_i, \Omega^{p-1}_{S_i}\otimes
\sO_{E}(-1))&=&\sum^{2}_{i=1}\chi(Z, \Omega^{p-1}_{Z}\otimes s_{i}^{*}\sO_{E}(-1))\\
 &=&\chi(Z,\Omega^{p-1}_{Z}\otimes(\bigoplus^{2}_{i=1}s_{i}^{*}\sO_{E}(-1)))\\
&=&\chi(Z,\Omega^{p-1}_{Z}\otimes \sN).
\end{eqnarray*}
Therefore, we get finally that
\begin{eqnarray*}
  \chi(\Omega^{p}_{E}(\log D)\otimes\sO_{E}(-1))&=&\sum^{2}_{i=1}\chi(S_i, \Omega^{p-1}_{S_i}\otimes \sO_{E}(-1))+
  \chi(\Omega_{E}^p\otimes \sO_{E}(-1)) \\
   &=& \chi(Z,\Omega^{p-1}_{Z}\otimes \sN)-\chi(Z,\Omega^{p-1}_{Z}\otimes \sN)  \\
   &=& 0.
\end{eqnarray*}
This completes the proof.

\end{proof}

{\itshape Proof of Claim \ref{formula}:}\\

Put $\sE:=\displaystyle \bigoplus^{2n+2}\sO_{\P^n}(1)$. Let
$\P=\mathbf{Proj}(S(\sE))\stackrel{p}{\to} \P^n$ be the associated projective vector bundle together with the invertible sheaf $\sM:=\sO(1)$. The sheaf of differential operators $\Sigma_{\sM}$ of $\sM$ with order
$\leq 1$ is defined by the following exact sequence:
\begin{displaymath}
\begin{array}{lllllllllllll}
0&\longrightarrow & \sO_{\P} & \longrightarrow &\Sigma_{\sM}&
\stackrel{q}{\longrightarrow} &T_{\P}&\longrightarrow& 0
\end{array}
\end{displaymath}
with the extension class $-c_1(\sM)\in
\textrm{Ext}^{1}(T_{\P},\sO_{\P})\simeq
H^1(\P,\Omega^{1}_{\P})$. For $i\in \{1,\cdots, 2n+2\}$,
we let $\lambda_i\in H^0(\P,\sM\otimes
p^{*}\sO_{\P^n}(-1))$ such that $\mathrm{P}_i=\{\lambda_i=0\}$ is
the divisor of $\P$ whose fiber under the projection $p$ is the
$i^{-th}$ coordinate hyperplane of the fiber of $\P$ under $p$.
Put $\sP=\sum^{2n+2}_{i=1}\mathrm{P}_i$.\\

Suppose $H$ is defined by the equation $\prod^{2n+2}_{i=1}F_i=0$
in $\P^n$. Then we associate to $H$ the section
$\sigma=\sum^{2n+2}_{i=1}F_i\cdot\lambda_i \in
H^0(\P,\sM)$. Put $\sZ=\{\sigma=0\}$. Since $H$ is normal
crossing, $\sZ$ is smooth in $\P$. Note that the section
$\sigma\in H^{0}(\P,\sM)$ defines the evaluation map
$j(\sigma): \Sigma_{\sM}\to \sM$.
\begin{lemma}\label{evaluation}
$j(\sigma)$ is surjective with kernel equal to
$T_{\P}(-\log \sZ)$. That is, the following exact
sequence holds:
\begin{displaymath}
\begin{array}{lllllllllllll}
0&\longrightarrow & T_{\P}(-\log \sZ)&\longrightarrow
&\Sigma_{\sM}&\stackrel{j(\sigma)}{\longrightarrow}&\sM&\longrightarrow&0
\end{array}
\end{displaymath}
\end{lemma}
\begin{proof} Since $\sZ$ is smooth, the system of equations $
\{\displaystyle \scriptstyle \frac{\partial}{\partial
x_1}(\sM_x)=\cdots=\frac{\partial}{\partial\lambda_{2n+2}}(\sM_x)=0
\}$ has no common solutions in $\P$. This means $j(\sigma)$ is
surjective. The local sections of $T_{\P}(-\log \sZ)$ are
the first order differential operators preserving $\sZ$. Then for
each open subset $U\subset \P$, one has
\begin{eqnarray*}
  T_{\P}(-\log \sZ)(U)&=&\{P\in \Sigma_{\sM}(U),\ P(\sigma)=\sigma\}\\
   &\simeq& \{P\in \Sigma_{\sM}(U),\ P(\sigma)=0\},
\end{eqnarray*}
Hence $\ker (\sigma)\simeq T_{\P}(-\log \sZ)$.

\end{proof}

Now define $\Sigma_{\sM}(-\log
\sP)=q^{-1}(T_{\P}(-\log \sP))\subset
\Sigma_{\sM}$.
\begin{lemma}\label{upstairs}
The following two exact sequences are exact:
\begin{displaymath}
\begin{array}{lllllllllll}
  0 &\longrightarrow& T_{\P}(-\log \sP+\sZ) & \longrightarrow & \Sigma_{\sM}(-\log \sP)&\longrightarrow & \sM  & \longrightarrow & 0  \\
\end{array}
\end{displaymath}
\begin{displaymath}
\begin{array}{lllllllllll}
0 &\longrightarrow& \sO_{\P}^{\oplus 2n+2} & \longrightarrow & \Sigma_{\sM}(-\log \sP) &\longrightarrow& p^{*}T_{\P^n}& \longrightarrow & 0 \\
\end{array}
\end{displaymath}
\end{lemma}
\begin{proof}
The proof for the first sequence is the same as
the one in Lemma \ref{evaluation}. The second sequence follows
from the defining sequence of $\Sigma_{\sM}$ and the Euler
sequences.

\end{proof}

Define $\Sigma^{0}_{\sE}=p_{*}\Sigma_{\sM}(-\log \sP)$.
Note that $\mathrm{R}^{1}p_{*}T_{\P}(-\log
\sP+\sZ)=\mathrm{R}^{1}p_{*}\sO_{\P}=0$. Then $p_*$ of
two short exact sequences in Lemma \ref{upstairs} gives the
following two short exact sequences of $\Sigma^{0}_{\sE}$:
\begin{corollary}\label{downstairs}
The following two sequences of vector bundles are exact:
\begin{displaymath}
\begin{array}{ccccccccc}
  0 &\longrightarrow& T_{\P^n}(-\log H) & \longrightarrow & \Sigma^{0}_{\sE}&\longrightarrow & \sE& \longrightarrow & 0  \\
\end{array}
\end{displaymath}
\begin{displaymath}
\begin{array}{ccccccccc}
 0 &\longrightarrow& \sO_{\P^n}^{\oplus 2n+2} & \longrightarrow &\Sigma^{0}_{\sE}&\longrightarrow&  T_{\P^n}& \longrightarrow & 0 \\
\end{array}
\end{displaymath}
\end{corollary}
Recall the following
\begin{lemma}\label{algebra lemma}
Let X be a compact complex manifold. Given a short exact sequence
of vector bundles on X:
\begin{displaymath}
\begin{array}{lllllllll}
0&\longrightarrow & E&\longrightarrow &F& \longrightarrow&
G&\longrightarrow &0,
\end{array}
\end{displaymath}
one has then a long exact sequence of the form:
\begin{displaymath}
\begin{array}{lllllllllllll}
0&\to & S^{k}E&\to &S^{k-1}E\otimes F& \to&\cdots&\to
&\wedge^{k}F&\to& \wedge^{k}G&\to &0.
\end{array}
\end{displaymath}
\end{lemma}

Applying the previous lemma to the dual of the first short exact sequence in
Corollary \ref{downstairs}, one obtains the following long exact
sequence:
\begin{displaymath}
\begin{array}{ccccccccccccccc}
0&\to&\Omega^{n-p}_{\P^n}(\log H)(-n-1)&\to &\wedge^{p}\Sigma^{0}_{\sE}&\to&\cdots&\to &S^{p}\sE& \to & 0.  \\
\end{array}
\end{displaymath}
\begin{lemma}\label{wedge product of Sigma}
$H^i(\P^n,S^{p-k}\sE\otimes\wedge^{k}\Sigma^{0}_{\sE})=0$,
for each $i>0$ and $0\leq k\leq p$.
\end{lemma}
\begin{proof}
It suffices to show $H^i(\P^n,\wedge^{k}\Sigma^{0}_{\sE})=0$. The sheaf $\wedge^{k}\Sigma^{0}_{\sE}$ has a filtration from the second exact sequence in Corollary \ref{downstairs},
\begin{displaymath}
\wedge^{k}\Sigma^{0}_{\sE}=\sF^0\supset
\sF^1\supset\cdots\supset\sF^k\supset\sF^{k+1}=0.
\end{displaymath}
with for each $0\leq \nu \leq k$,
\begin{eqnarray*}
  \textrm{Gr}^{\nu}=\sF^{\nu}/\sF^{\nu+1} &=& \wedge^{\nu}(\oplus^{2n+2}\sO_{\P^n}\otimes \wedge^{k-\nu}T_{\P^n}) \\
   &=& \bigoplus^{2n+2\choose \nu}\wedge^{k-\nu}T_{\P^n} \\
   &=& \bigoplus^{2n+2\choose \nu}\Omega^{n+\nu-k}_{\P^n}(n+1).
\end{eqnarray*}
By the Bott vanishing theorem,
$H^i(\P^n,\Omega^{n+\nu-k}_{\P^n}(n+1))=0$, which implies
$H^i(\P^n,\sF^{\nu})=0$. In particular,
$H^i(\P^n,\wedge^{k}\Sigma^{0}_{\sE})=0$.

\end{proof}

\begin{proposition}\label{key computation}
$\chi(\P^n,\Omega^{n-p}_{\P^n}(\log H)(-n-1))=(-1)^p
{n \choose p}^2$.
\end{proposition}
\begin{proof}
Taking the Euler characteristic of the long exact sequence
below the Corollary \ref{downstairs}, we have
\begin{eqnarray*}
   \chi(\Omega^{n-p}_{\P^n}(\log H)(-n-1)) &=&
   \sum^{p}_{k=0}(-1)^{k}\chi(S^{p-k}\sE\otimes\wedge^{k}\Sigma^{0}_{\sE}).
\end{eqnarray*}
Again by the filtration of the sheaf
$\wedge^{k}\Sigma^{0}_{\sE}$ in the proof of Lemma \ref{wedge product of Sigma}, one computes that
\begin{eqnarray*}
  \chi({\wedge^{k}\Sigma^{0}_{\sE}}) &=& \displaystyle \sum^{k}_{i=0}\chi(\textrm{Gr}_i) \\
     &=&\displaystyle \sum^{k}_{i=0} \displaystyle
     \sum^{i}_{j=0}(-1)^{i+j}{2n+2 \choose j}{n+1 \choose
     k-i}\chi(\sO_{\P^n}(k-i)).
\end{eqnarray*}
Therefore,
\begin{eqnarray*}
  \chi(\Omega^{n-p}_{\P^n}(\log B)(-n-1)) &=&   \sum^{p}_{i=0}  \sum^{p-i}_{j=0}
\sum^{p-i-j}_{k=0}\textstyle (-1)^{i+k}{2n+1+i \choose i}{n+1
\choose j}{n+i+j \choose i+j}{2n+2 \choose
     p-i-j-k} \\
   &=&  \sum^{p}_{i=0} \sum^{p-i}_{j=0}\textstyle(-1)^{i}{2n+1+i \choose i}{n+1 \choose
     j}{n+i+j \choose i+j}{2n+1 \choose p-i-j} \\
   &=&   \sum^{p}_{l=0}   \sum^{l}_{i=0}\textstyle(-1)^{i}{2n+1+i \choose i}{n+1 \choose l-i}{n+l \choose n}{2n+1 \choose p-l}\\
    &=&  \sum^{p}_{l=0}\textstyle(-1)^{l}{n+l \choose l}^{2}{2n+1 \choose p-l} \\
    &=&  (-1)^{p}{n \choose p}^2.
\end{eqnarray*}

\end{proof}

\subsection{Calculations in Jacobian ring}\label{subsection:Jacobian ring}
In this subsection we will prove the upper bound of the dimension of the characteristic variety $C_{1,a}$ claimed in Proposition \ref{prop:dim of characteristic variety}. Recall the definitions of $X$ and $Y$ from Section \ref{subsection:Kummer cover}.
We want to compute the Hodge structure and the Higgs maps on $X$. Since the Hodge structure on $X$ is determined by that on $Y$, we first analyze the Hodge structure on $Y$. In order to do that, we use the tool of Jacobian ring. It is constructed as follows.  In the polynomial ring $\C[\mu_0,\cdots, \mu_{n}, y_0,\cdots, y_{2n+1}]$, consider the polynomial
$$
F=\mu_0F_0+\cdots +\mu_{n}F_{n}
$$
where
\begin{equation}\notag
\begin{split}
&F_0:=y_{n+1}^2-(y_0^2+y_1^2+\cdots +y_n^2),\\
&F_i:=y_{n+i+1}^2-(y_0^2+a_{1i}y_1^2+\cdots +a_{ni}y_n^2), \ 1\leq i\leq   n.
\end{split}
\end{equation}

Let $J=<\frac{\partial F}{\partial \mu_i}, \frac{\partial F}{\partial y_j} \mid 0\leq i\leq n, 0\leq j\leq 2n+1>$ be the ideal of $\C[\mu_0,\cdots, \mu_{n},$ $ y_0,\cdots, y_{2n+1}]$ generated by the partial derivatives of $F$. Define the Jacobian ring to be
$$
R:=\C[\mu_0,\cdots, \mu_{n}, y_0,\cdots, y_{2n+1}]/J
$$

There is a natural bigrading on the polynomial ring $\C[\mu_0,\cdots, \mu_{n}, y_0,\cdots, y_{2n+1}]$, that is: the $(p,q)-$part $\C[\mu_0,\cdots, \mu_{n}, y_0,\cdots, y_{2n+1}]_{(p,q)}$ is linearly spanned by the monomials $\Pi_{i=0}^{n}\mu_i^{a_i}\Pi_{j=0}^{2n+1}y_j^{b_j}$ with $\sum_{i=0}^{n}a_i=p$, $\sum_{j=0}^{2n+1}b_j=q$. Since the ideal $J$ is a homogeneous ideal, there is a natural induced bigrading of $R=\C[\mu_0,\cdots, \mu_{n}, y_0,\cdots, y_{2n+1}]/J$, written as  $R=\oplus_{p,q\geq 0} R_{(p,q)}$.

The group $N=\oplus_{j=0}^{2n+1}\F_2$ acts on $R$ through $y_0,\cdots, y_{2n+1}$. Explicitly, $\forall g=(a_j)\in N$, we define the action of $g$ on $R$  by
\begin{equation}\notag
\begin{split}
g\cdot y_j&:=(-1)^{a_j}y_j, \ \ \forall \ 0\leq j\leq 2n+1.\\
g\cdot \mu_i&:=\mu_i,\ \ \forall \ 0\leq i\leq n.
\end{split}
\end{equation}

It is obviously that the action of $N$ on $R$ preserves the bigrading.
Let $R_{(p,q)}^{N}$ be the $N-$invariant part of $R_{(p,q)}$, then we have the decomposition of the $N-$invariant subring: $R^{N}=\oplus_{p,q\geq 0} R_{(p,q)}^{N}$.

Recall $\V_{(1)}$ is the $(-1)$-eigen $\Q$-PVHS associated to the  family $\mathcal{X}_{AR}\rightarrow \mathfrak{M}_{AR}$. The following proposition gives an identification of the Higgs map associated to $\V_{(1)}$ and the Jacobian ring multiplication.
\begin{proposition}\label{prop:identification of Hodge structures with Jacobian ring}
\begin{itemize}
\item[(1)] $\forall \ 0\leq q\leq n$, $H^{n-q,q}(X)_{(1)}\simeq R^{N}_{(q,2q)}$;
\item[(2)] $\forall \ 0\leq q\leq n$,  we have a commutative diagram
\begin{equation}\notag
\begin{array}{ccc}
  T_{\mathfrak{M}_{AR},a}\otimes H^{n-q,q}(X)_{(1)} & \xrightarrow{\theta^{n-q,q}} & H^{n-q-1,q+1}(X)_{(1)} \\
   \downarrow{\simeq}&  & \downarrow{\simeq} \\
  R^{N}_{(1,2)}\otimes R^{N}_{(q,2q)} &\xrightarrow{} & R^{N}_{(q+1,2q+2)}.
\end{array}
\end{equation}
\end{itemize}
Here $X$ is the fiber over $a\in \mathfrak{M}_{AR}$, and the lower horizontal arrow is the ring  multiplication map.
\end{proposition}

\begin{proof}
(1) Let $N_1=Ker (\oplus_{j=0}^{2n+1}\F_2\xrightarrow{\sum}\F_2)$ be the  subgroup of $N=\oplus_{j=0}^{2n+1}\F_2$. By Corollary 2.5 in \cite{T} and its proof, we can see that  $\forall \ 0\leq q\leq n$, $H^{n-q,q}(Y)^{N_1}_{(1)}\simeq R^{N_1}_{(q,2q)(0)}=R^N_{(q,2q)}$. By Proposition \ref{prop:pure Hodge structure of X}, $\forall \ 0\leq q\leq n$, $H^{n-q,q}(X)_{(1)}\simeq H^{n-q,q}(Y)^{N_1}_{(1)}$, these two isomorphisms together gives (1).

(2) follows from (1) and Proposition  2.6 in \cite{T}.
\end{proof}



Next we will present $\C$-bases of $R^{N}_{(1,2)}$ and $R^{N}_{(2,4)}$.
  Note that the ideal $J$ is generated by the following elements:
\begin{equation}\notag
\begin{split}
\frac{\partial F}{\partial \mu_0}&=y_{n+1}^2-(y_0^2+y_1^2+\cdots +y_n^2);\\
\frac{\partial F}{\partial\mu_i}&=y_{n+i+1}^2-(y_0^2+a_{1i}y_1^2+\cdots+a_{ni}y_n^2), \ 1\leq i\leq n;\\
-\frac{\partial F}{2\partial y_0}&=y_0(\mu_0+\mu_1+\cdots+\mu_{n});\\
-\frac{\partial F}{2\partial y_i}&=y_i(\mu_0+a_{i1}\mu_1+\cdots+a_{in}\mu_{n}), \ 1\leq i\leq n;\\
\frac{\partial F}{2\partial y_{n+i+1}}&=\mu_iy_{n+i+1}, \ 0\leq i\leq n.
\end{split}
\end{equation}

 By the relations above, we can see easily that
$$
R^{N}_{(1,2)}=\C<\mu_iy_j^2 \mid 0\leq i,j \leq n>,
$$
where $\C<\mu_iy_j^2 \mid 0\leq i,j \leq n>$ means the linear subspace of $R$ spanned by elements in the set $\{\mu_iy_j^2 \mid 0\leq i,j \leq n\}$. Similarly, one finds  that
$$
R^{N}_{(2,4)}=\C<\mu_iy_j^2\mu_py_q^2 \mid 0\leq i,j,p,q\leq n>.
$$

In order to obtain bases from  $\{\mu_iy_j^2 \mid 0\leq i,j \leq n\}$ and $\{\mu_iy_j^2\mu_py_q^2 \mid 0\leq i,j,p,q\leq n\}$, we study the relations in $R$.


In order to write the relations  more symmetrically,
we define
$$
a_{i0}=a_{0j}=1, \ \ \forall \ 0\leq i, j\leq n.
$$


Then we find easily that  the following relations hold in $R$:
\begin{equation}\label{equation:basic relations in R}
 \left\{
  \begin{array}{ll}
    \sum_{j=0}^{n}a_{ji}\mu_i y_j^2=0, & \hbox{$\forall \ 0\leq i\leq n$;} \\
    \sum_{i=0}^{n}a_{ji}\mu_i y_j^2=0, & \hbox{$\forall \ 0\leq j\leq n$.}
  \end{array}
\right.
\end{equation}

Note that all the discussion above depends on the parameter $a:=(a_{ij})\in M(n\times n,\C)$.
From  these basic relations (\ref{equation:basic relations in R}), we can get  some other useful  relations.
\begin{lemma}\label{lemma:reduction}
For a generic parameter $a\in M(n\times n,\C)$, the following relations hold in $R$:
\begin{itemize}
 \item[(R1)]$ (\sum_{j=0}^{n}a_{ji}\mu_i y_j^2)\mu_p y_q^2=0$, $\forall \  0\leq i,p, q\leq n$.
  \item[(R2)] $(\sum_{i=0}^{n}a_{ji}\mu_i y_j^2)\mu_p y_q^2=0$, $ \forall \ 0\leq p, q,j \leq n$.
\item[(R3)] $  \mu_py_q^2\mu_iy_q^2=\sum_{j=1,j\neq q}^{n}\frac{a_{0p}a_{ji}-a_{0i}a_{jp}}{a_{0i}a_{qp}-a_{0p}a_{qi}}\mu_p y_q^2\mu_iy_j^2$, $ \forall \ 0\leq p\neq i\leq n$, $ \forall \ 1\leq q\leq n$.
\item[(R4)]$  \mu_py_j^2\mu_py_q^2=\sum_{i=1,i\neq p}^{n}\frac{a_{q0}a_{ji}-a_{j0}a_{qi}}{a_{j0}a_{qp}-a_{q0}a_{jp}}\mu_i y_j^2\mu_py_q^2$, $ \forall \ 1\leq p\leq n$, $ \forall \ 0\leq j\neq q\leq n$.
\end{itemize}
\end{lemma}
\begin{proof}
The relations $(R1)$ and $(R2)$ follow obviously from the basic relations (\ref{equation:basic relations in R}).

Given $0\leq p\neq i\leq n$, $1\leq q\leq n$, in order to prove $(R3)$, we consider the element $\mu_p\mu_iy_0^2y_q^2$. By relations $(R1)$, we have
\begin{equation}\notag
\begin{split}
 a_{0p}\mu_p\mu_iy_0^2y_q^2&=a_{0p}\mu_py_0^2\mu_iy_q^2=-\sum_{j=1}^{n}a_{jp}\mu_py_j^2\mu_iy_q^2,\\
a_{0i}\mu_p\mu_iy_0^2y_q^2&=a_{0i}\mu_iy_0^2\mu_py_q^2=-\sum_{j=1}^na_{ji}\mu_iy_j^2\mu_py_q^2.
 \end{split}
\end{equation}
Let $a_{0i}$ times the first identity and $a_{0p}$ times the second identity, we get
$$
\sum_{j=1}^{n}a_{0i}a_{jp}\mu_py_j^2\mu_iy_q^2=\sum_{j=1}^na_{0p}a_{ji}\mu_iy_j^2\mu_py_q^2.
$$
From this equality we get
\begin{equation}\notag
\begin{split}
(a_{0i}a_{qp}-a_{0p}a_{qi})\mu_py_q^2\mu_iy_q^2&=a_{0i}a_{qp}\mu_py_q^2\mu_iy_q^2-a_{0p}a_{qi}\mu_iy_q^2\mu_p y_q^2\\
&=\sum_{j=1,j\neq q}^na_{0p}a_{ji}\mu_iy_j^2\mu_py_q^2-\sum_{j=1,j\neq q}^{n}a_{0i}a_{jp}\mu_py_j^2\mu_iy_q^2\\
&=\sum_{j=1,j\neq q}^n(a_{0p}a_{ji}-a_{0i}a_{jp})\mu_py_q^2\mu_iy_j^2.
\end{split}
\end{equation}

Since for a generic parameter $a\in \C^{n^2}$, $a_{0i}a_{qp}-a_{0p}a_{qi}=a_{qp}-a_{qi}\neq 0$,  we get $(R3)$.

The relations $(R4)$ can be proved similarly.
\end{proof}
Now we can determine $\C$-bases of $R^{N}_{(1,2)}$ and $R^{N}_{(2,4)}$.
\begin{proposition} \label{prop:bases}
\begin{itemize}
\item[(1)] For any parameter $  a\in \C^{n^2}$, $R^{N}_{(1,2)}$ has a $\C$-basis $\{\mu_iy_j^2 \mid 1\leq i, j\leq n\}$.
\item[(2)]For a generic parameter $a\in \C^{n^2}$, $R^{N}_{(2,4)}$ has a $\C$-basis $\{\mu_iy_j^2\mu_py_q^2 \mid 1\leq i<p\leq n, 1\leq j<q\leq n.\}$
    \item[(3)] For any parameter $  a\in \C^{n^2}$, $\forall q\geq 1$, the multiplication map $Sym^q R^{N}_{(1,2)}\rightarrow R^{N}_{(q,2q)}$ is surjective.
\end{itemize}
\end{proposition}

\begin{proof}
(1): By the basic relations (\ref{equation:basic relations in R}), we have
$$
R^{N}_{(1,2)}=\C<\mu_iy_j^2 \mid 1\leq i, j\leq n>.
$$
 Proposition \ref{prop:reduce to V_1}  and Proposition \ref{prop:identification of Hodge structures with Jacobian ring}   imply that the dimension of the $\C$-linear space $R^{N}_{(1,2)}$ is $n^2$. So we get that $\{\mu_iy_j^2 \mid 1\leq i, j\leq n\}$ is a $\C$-basis of $R^{N}_{(1,2)}$. This proves (1).

(2): Similarly as (1),  the dimension of the $\C$-linear space $R^{N}_{(2,4)}$ is ${n\choose 2}^2$, and $R^{N}_{(2,4)}$ is linearly spanned by $\{\mu_iy_j^2\mu_py_q^2 \mid 0\leq i, j, p, q\leq n\}$, so it suffices to show that $\forall \ 0\leq i, j, p, q\leq n$,
$$
 \mu_iy_j^2\mu_py_q^2 \in \C<\mu_iy_j^2\mu_py_q^2 \mid 1\leq i<p\leq n, 1\leq j<q\leq n>.
$$
This can be proved using the relations $R(1)-R(4)$ in Lemma \ref{lemma:reduction}. Next we show $\mu_0y_0^2\mu_0y_0^2 \in \C<\mu_iy_j^2\mu_py_q^2 \mid 1\leq i<p\leq n, 1\leq j<q\leq n>$ by the following steps to illustrate the ideas.

Step 1: By relations $(R1)$, $\mu_0y_0^2\mu_0y_0^2 \in \C<\mu_0y_0^2\mu_0y_q^2 \mid  1\leq q\leq n>$.

Step 2: $\forall 1\leq q\leq n$, by relations $(R2)$, $\mu_0y_0^2\mu_0y_q^2 \in \C<\mu_0y_0^2\mu_py_q^2 \mid  1\leq p\leq n>$.

Step 3: $\forall 1\leq p, q\leq n$,  by relations $(R2)$, $\mu_0y_0^2\mu_py_q^2 \in \C<\mu_iy_0^2\mu_py_q^2 \mid  1\leq i\leq n>$.

Step 4: $\forall 1\leq i, p, q\leq n$,   if $i=p$, by relations $(R4)$, $\mu_iy_0^2\mu_py_q^2 \in \C<\mu_{i_1}y_0^2\mu_py_q^2 \mid  1\leq i_1\leq n, i_1\neq p>$.

Step 5: $\forall 1\leq q\leq n$, $\forall 1\leq i\neq p\leq n$, by relations $(R1)$, $\mu_iy_0^2\mu_py_q^2 \in \C<\mu_{i}y_j^2\mu_py_q^2 \mid  1\leq j \leq n>$.

Step 6: $\forall 1\leq q\leq n$, $\forall 1\leq i\neq p\leq n$, by relations $(R3)$, $\mu_iy_q^2\mu_py_q^2 \in \C<\mu_{i}y_j^2\mu_py_q^2 \mid  1\leq j \leq n, j\neq q>$.

After these six steps, we have shown $\mu_0y_0^2\mu_0y_0^2 \in \C<\mu_iy_j^2\mu_py_q^2 \mid 1\leq i<p\leq n, 1\leq j<q\leq n>$. Other cases can be treated similarly.

(3) follows directly from the definition of the Jacobian ring.
\end{proof}


Given $\alpha \in R^{N}_{(1,2)}$, we can expand it under the basis above:
$$
\alpha =\sum_{1\leq i, j\leq n}\lambda_{ij}\mu_iy_j^2,
$$
with $(\lambda_{ij})\in M(n\times n, \C)$.

Since $\{\mu_iy_j^2\mu_py_q^2 \mid 1\leq i<p\leq n, 1\leq j<q\leq n\}$ is a basis of $R^{N}_{(2,4)}$, we have the expression
$$
\alpha^2=\sum_{1\leq i<p\leq n, 1\leq j<q\leq n}f_{ijpq}\mu_iy_j^2\mu_py_q^2.
$$
Obviously each $f_{ijpq}$ is a homogeneous quadratic polynomial of $\lambda_{ij}(1\leq i, j\leq n)$, with coefficients being rational functions of the parameters $a_{ij}(1\leq i, j\leq n)$. As for the information of these $f_{ijpq}$, we have the following proposition.
\begin{proposition}\label{proposition: shape of f_ijpq}The following statements hold:
\begin{itemize}
\item[(1)]
$\forall \ 1\leq i<p\leq n$, $\forall \ 1\leq j<q\leq n$, we have
\begin{equation}\notag
\begin{split}
f_{ijpq}&=c_{ijpq}^{ijij}\lambda_{ij}^2+c_{ijpq}^{iqiq}\lambda_{iq}^2+c_{ijpq}^{pjpj}\lambda_{pj}^2+c_{ijpq}^{pqpq}\lambda_{pq}^2\\
&+c_{ijpq}^{ijiq}\lambda_{ij}\lambda_{iq}+c_{ijpq}^{pjpq}\lambda_{pj}\lambda_{pq}+c_{ijpq}^{ijpj}\lambda_{ij}\lambda_{pj}+c_{ijpq}^{iqpq}\lambda_{iq}\lambda_{pq}\\
&+ c_{ijpq}^{ijpq}\lambda_{ij}\lambda_{pq}+ c_{ijpq}^{iqpj}\lambda_{iq}\lambda_{pj},
\end{split}
\end{equation}
where each of the ten coefficients $c_{ijpq}^{ijij},\cdots, c_{ijpq}^{iqpj}$ is a nonzero rational function of $a_{ji}, a_{qi}, a_{jp}, a_{qp}$.
\item[(2)] Notations as in (1). $\forall \ 2\leq j<q\leq n$, let $R_{jq}$ be the following resultant:
\begin{displaymath}
R_{jq}:=det \left(
  \begin{array}{cccc}
    c_{112q}^{1q1q} & 0 & c_{1j2q}^{1q1q} & 0 \\
    c_{112q}^{1q2q} &  c_{112q}^{1q1q} & c_{1j2q}^{1q2q} & c_{1j2q}^{1q1q} \\
    c_{112q}^{2q2q} & c_{112q}^{1q2q} & c_{1j2q}^{2q2q} & c_{1j2q}^{1q2q} \\
    0 & c_{112q}^{2q2q} & 0 & c_{1j2q}^{2q2q} \\
  \end{array}
\right).
\end{displaymath}
Then $R_{jq}$ is a nonzero rational function of $a_{ij}$  $(1\leq i, j\leq n)$.
\item[(3)] Notations as in (1). $\forall \ 2\leq i<p\leq n$, let $Q_{ip}$ be the following resultant:
\begin{displaymath}
Q_{ip}:=det \left(
  \begin{array}{cccc}
    c_{11p2}^{p1p1} & 0 & c_{i1p2}^{p1p1} & 0 \\
    c_{11p2}^{p1p2} &   c_{11p2}^{p1p1} &  c_{i1p2}^{p1p2} & c_{i1p2}^{p1p1} \\
    c_{11p2}^{p2p2} & c_{11p2}^{p1p2} &  c_{i1p2}^{p2p2} & c_{i1p2}^{p1p2} \\
    0 & c_{11p2}^{p2p2} & 0 & c_{i1p2}^{p2p2} \\
  \end{array}
\right).
\end{displaymath}
Then $Q_{ip}$ is a nonzero rational function of $a_{ij}$  $(1\leq i, j\leq n)$.
\end{itemize}
\end{proposition}
\begin{proof}
(1): Examine the proof of Proposition \ref{prop:bases}, one can get the following fact:

$\forall 1\leq i_1,i_2, j_1,j_2\leq n$, if $\{(i_1,j_1),(i_2,j_2)\}\nsubseteqq \{(i,j),(i,q),(p,j),(p,q)\}$, then when we express  $\mu_{i_1}y_{j_1}^2\mu_{i_2}y_{j_2}^2$ as a linear combination of the basis $\{\mu_iy_j^2\mu_py_q^2 \mid 1\leq i<p\leq n, 1\leq j<q\leq n\}$, the coefficient before $\mu_{i}y_{j}^2\mu_{p}y_{q}^2$ is zero.

Using this observation, we find $f_{ijpq}$ has the required expression. That  each of the ten coefficients $c_{ijpq}^{ijij},\cdots, c_{ijpq}^{iqpj}$ is a nonzero rational function of $a_{ji}, a_{qi}, a_{jp}, a_{qp}$ follows from an explicit computation. Explicitly, we have:
\begin{equation}\notag
\begin{split}
& c_{ijpq}^{ijpq}=2; \ c_{ijpq}^{iqpj}=2; \\
&c_{ijpq}^{ijiq}=\frac{2(a_{jp}-a_{qp})}{a_{qi}-a_{ji}}; \  c_{ijpq}^{pjpq}=\frac{2(a_{ji}-a_{qi})}{a_{qp}-a_{jp}};\\
& c_{ijpq}^{ijpj}=\frac{2(a_{qi}-a_{qp})}{a_{jp}-a_{ji}}; \ c_{ijpq}^{iqpq}=\frac{2(a_{ji}-a_{jp})}{a_{qp}-a_{qi}};\\
& c_{ijpq}^{ijij}=\frac{1}{a_{ji}}\cdot \frac{a_{jp}-1}{a_{ji}-1}\cdot (\frac{a_{jp}(a_{qp}-a_{qi})}{a_{jp}-a_{ji}}-a_{qp})+\frac{a_{qi}}{a_{ji}}\cdot \frac{a_{qp}-a_{jp}}{a_{qi}-a_{ji}};\\
& c_{ijpq}^{iqiq}=\frac{1}{a_{qi}}\cdot \frac{a_{qp}-1}{a_{qi}-1}\cdot (\frac{a_{qp}(a_{jp}-a_{ji})}{a_{qp}-a_{qi}}-a_{jp})+\frac{a_{ji}}{a_{qi}}\cdot \frac{a_{jp}-a_{qp}}{a_{ji}-a_{qi}};\\
& c_{ijpq}^{pjpj}=\frac{1}{a_{jp}}\cdot \frac{a_{ji}-1}{a_{jp}-1}\cdot (\frac{a_{ji}(a_{qi}-a_{qp})}{a_{ji}-a_{jp}}-a_{qi})+\frac{a_{qp}}{a_{jp}}\cdot \frac{a_{qi}-a_{ji}}{a_{qp}-a_{jp}};\\
& c_{ijpq}^{pqpq}=\frac{1}{a_{qp}}\cdot \frac{a_{qi}-1}{a_{qp}-1}\cdot (\frac{a_{qi}(a_{ji}-a_{jp})}{a_{qi}-a_{qp}}-a_{ji})+\frac{a_{jp}}{a_{qp}}\cdot \frac{a_{ji}-a_{qi}}{a_{jp}-a_{qp}}.
\end{split}
\end{equation}

(2) and (3): From the above  explicit computation of the coefficients $c_{ijpq}^{ijij},\cdots, c_{ijpq}^{iqpj}$, we can  see that $R_{jq}$ and  $Q_{ip}$ are nonzero rational functions of $a_{ij}$  $(1\leq i, j\leq n)$.

\end{proof}

Now we can complete the proof of Proposition \ref{prop:dim of characteristic variety}.
Recall $C_{1,a}$ is the first characteristic variety of $\V_{(1)}$ at $a\in \mathfrak{M}_{AR}$. By Proposition \ref{prop:identification of Hodge structures with Jacobian ring}, it is easy to see that
$$
C_{1,a}\simeq C^{'}_{1,a}:=\{[\alpha]\in \P(R^{N}_{(1,2)})\mid \alpha^2=0 \in R^{N}_{(2,4)}  \},
$$
where $\P(R^{N}_{(1,2)})$ is the projectification of the $\C$-linear space $R^{N}_{(1,2)}$, and $[\alpha]$ means the class represented by an element $\alpha\in R^{N}_{(1,2)}$. For each $a\in \C^{n^2}$, $C^{'}_{1,a}$ is a closed subvariety in the projective space $\P(R^{N}_{(1,2)})$.

For later use, we state the following elementary lemma.
\begin{lemma}\label{lemma:dimension lemma}
Let $V=\C^p$ be an affine space with coordinates $x_1,\cdots, x_p$. Then the following holds:
\begin{itemize}
\item[(1)] Given
\begin{displaymath}
\begin{array}{ccc}
  V_1 & \subset & V \\
  \cup &  & \cup \\
  X_1 &  & X
\end{array}
\end{displaymath}
where
\begin{itemize}
\item $V_1$ is the subspace of $V$ defined by $x_1=0$;
\item $X_1$ is the subvariety of $V_1$ defined by the simultaneous vanishing of the polynomials $f_j(x_2,\cdots, x_p)$, $2\leq j\leq q$;
\item $X$ is the subvariety of $V$ defined by the simultaneous vanishing of the polynomials  $f_1(x_1,\cdots, x_p)$, $f_j(x_2,\cdots, x_p)$, $2\leq j\leq q$, where  $f_1(x_1,\cdots, x_p)=ax_1^2+g_1(x_2,\cdots, x_p)x_1+h_1(x_2,\cdots, x_p)$, with  $0\neq a\in \C$.
\end{itemize}
Then we have $dim X\leq dim X_1$.
\item[(2)]
Given
\begin{displaymath}
\begin{array}{ccc}
  V_2 & \subset & V \\
  \cup &  & \cup \\
  X_2 &  & X
\end{array}
\end{displaymath}
where
\begin{itemize}
\item  $V_2$ is the subspace of $V$ defined by $x_1=x_2=0$;
\item $X_2$ is the subvariety of $V_2$ defined by the simultaneous vanishing of the polynomials $f_j(x_3,\cdots, x_p)$, $3\leq j\leq q$;
\item $X$ is the subvariety of $V$ defined by the simultaneous vanishing of the polynomials $f_1(x_1,\cdots,x_p)$, $f_2(x_1,\cdots,x_p)$, $f_j(x_3,\cdots, x_p)$, $3\leq j\leq q$, and for $i=1,2$, $f_i(x_1,\cdots,x_p)=a_ix_1^2+ b_ix_1x_2+c_ix_2^2+x_1g_i(x_3,\cdots,x_p)+x_2h_i(x_3,\cdots,x_p)$ $+$ $r_i(x_3,\cdots,x_p)$, with $a_i,b_i,c_i\in \C$ such that the following resultant is not zero:
\begin{displaymath}
 det \left(
  \begin{array}{cccc}
    a_1 & 0 & a_2 & 0 \\
    b_1 &   a_1 &  b_2 & a_2 \\
    c_1 &  b_1  &  c_2 & b_2 \\
    0 &  c_1  & 0 & c_2  \\
  \end{array}
\right)\neq 0.
\end{displaymath}
Then we have $dim X\leq dim X_2$.
\end{itemize}
\end{itemize}
\end{lemma}
\begin{proof}
The proof is direct. Since in each case we can consider the natural projections:
$\pi_i: V\rightarrow V_i$ $(i=1,2)$ and the conditions guarantee that $\forall x\in V_i$, the dimension of $X\cap \pi^{-1}(x)$ is either empty or a zero dimensional variety.
\end{proof}

Proposition \ref{prop:dim of characteristic variety} follows directly from
\begin{proposition}\label{prop:Jacobian ring: dim upper bound}
If $n\geq 2$, then for generic $a\in \C^{n^2}$, we have dim$C^{'}_{1,a} \leq 2$.
\end{proposition}
\begin{proof}
By Proposition \ref{proposition: shape of f_ijpq}, we can choose a generic parameter $a=(a_{ij})\in \C^{n^2}$, such that Proposition \ref{prop:bases} holds, and $\forall 1\leq i<p\leq n$, $\forall 1\leq j<q\leq n$, each of the rational functions $c_{ijpq}^{pqpq}$, $R_{jq}$, $Q_{ip}$ takes nonzero value at the point $a$. We only need to show that at this point $a$, $dim C^{'}_{1,a}\leq 2$. In the following, we fix this parameter $a$.

Under the basis $\mu_jy_j^2$ $(1\leq i, j\leq n)$, we identify $R^{N}_{(1,2)}$ with $\C^{n^2}$, and we view $(\lambda_{ij})$  $(1\leq i, j\leq n)$ as the coordinates on the affine space  $\C^{n^2}$.

It is obvious that the cone in $R^{N}_{(1,2)}=\C^{n^2}$ corresponding to $C^{'}_{1,a}$ is the variety $\tilde{X}\subset \C^{n^2}$ defined by the simultaneously vanishing of the ${n\choose 2}^2$ homogeneous  quadratic polynomials $f_{ijpq}$ $(1\leq i<p\leq n, 1\leq j<q\leq n)$. Define $X\subset \C^{n^2}$ by the simultaneously vanishing of the following $n^2-3$ homogeneous  quadratic polynomials:
\begin{displaymath}
\begin{split}
  &f_{i,j,i+1,j+1} \ \ (1\leq i, j\leq n-1);\\
  &f_{1,1,2,q} \ \ (3\leq q\leq n); \\
  &f_{1,1,p,2} \ \ (3\leq p\leq n).
\end{split}
\end{displaymath}

Since $dim C^{'}_{1,a}=$$dim \tilde{X}-1$ and $\tilde{X}\subset X$, it suffices to show $dim X\leq 3$.

In order to prove $dim X\leq 3$ using Lemma \ref{lemma:dimension lemma}, we give a filtration of $\C^{n^2}$ by affine spaces and define a subvariety in each of these affine spaces, i.e. we want to get the following  diagram:
\begin{displaymath}
\begin{array}{ccccccccc}
  V_1 & \subset & V_2 & \subset & \cdots & \subset & V_t & = & \C^{n^2} \\
  \cup &  & \cup &  &  &  & \cup &  & \cup \\
  X_1 &  & X_2 &  & \cdots &  & X_t & = & X
\end{array}
\end{displaymath}
with $t=(n-1)^2$.


Recall $(\lambda_{ij})$ $(1\leq i, j\leq n)$ are coordinates on $\C^{n^2}$.  First give a filtration of the index set $S:=\{(i,j) \mid 1\leq i, j\leq n\}$ as follows
$$
S_1\subset S_2\subset\cdots \subset S_t=S,
$$
where we define inductively
\begin{itemize}
\item $\forall 1\leq p \leq n-1$, $S_p:=\{(i,j)\in S\mid 1\leq i\leq 2, 1\leq j\leq p+1\}$;
\item $\forall k\geq 1$, $\forall k(n-1)+1\leq p\leq (k+1)(n-1), S_p:=S_{k(n-1)}\cup \{(k+2,j)\mid 1\leq j\leq p-k(n-1)+1\}$.
\end{itemize}
Now $\forall 1\leq p\leq t$, define the affine space
$$
V_{p}:=\{(\lambda_{ij})\in  \C^{n^2}\mid \lambda_{ij}=0, \forall (i,j)\notin S_p\}.
$$
Then $\forall 1\leq p\leq t$, we define  $X_p\subset V_p$  by the simultaneous vanishing of the polynomials in $\mathcal{F}_p$, where $\mathcal{F}_p$ is the set of polynomials defined inductively   as follows:
\begin{itemize}
\item $\mathcal{F}_1:=\{f_{1122}\}$;
\item $\forall 2\leq p\leq n-1$, $\mathcal{F}_p:=\mathcal{F}_{p-1}\cup \{f_{1,1,2,p+1}, f_{1,p,2,p+1}\}$;
\item $\forall k\geq 1$, $\forall k(n-1)+1\leq p\leq (k+1)(n-1)$, \newline $\mathcal{F}_p:= \mathcal{F}_{p-1}\cup \{f_{1,1,k+2,2}, f_{k+1,p-k(n-1),k+2,p-k(n-1)+1}\}$.
\end{itemize}
By Proposition \ref{proposition: shape of f_ijpq}, each $f_{ijpq}$ is a polynomial of the four variables $\lambda_{ij},\lambda_{iq},\lambda_{pi},\lambda_{pq}$,  so  $X_{p}\subset V_{p}$ is well defined.
According to our choice of the  parameter $a$, $\forall 1\leq i<p\leq k-1$, $\forall 1\leq j<q\leq n$, each of the rational functions $c_{ijpq}^{pqpq}$, $R_{jq}$, $Q_{ip}$ takes nonzero value at $a$. Then  a direct verification shows that $\forall 1\leq p\leq t-1$, the diagram
\begin{displaymath}
\begin{array}{ccc}
  V_{p-1} & \subset & V_{p} \\
  \cup &  & \cup \\
  X_{p-1} &  & X_{p}
\end{array}
\end{displaymath}
satisfies the conditions in (1) or (2) of Lemma \ref{lemma:dimension lemma}, hence we get $dim X=dim X_t$ $\leq$ $dim X_{t-1}$ $\leq \cdots \leq $ $ dim X_1$. By definition, one can see that $X_1$ is a hypersurface in $\C^4$ defined by a nonzero polynomial, so $dim X_1\leq 3$, and finally we get $dim X\leq 3$.
\end{proof}


\begin{thebibliography}{X-X00}
\addcontentsline{toc}{chapter}{Bibliography}

\bibitem{A'Campo}
A{'}Campo, N.; {\it Tresses, monodromie et le groupe symplectique}, Comment. Math. Helv. 54(2),
318�27, 1979.

\bibitem{ACT}
Allcock,D; Carlson,J.; Toledo,D; {\it The complex hyperbolic
geometry of the moduli space of cubic surfaces}.  J. Algebraic
Geometry.  11 (2002),  no. 4, 659-724.



\bibitem{Bo}
Borcherds,R; {\it The moduli space of Enriques surfaces and the fake
monster Lie superalgebra}, Topology 35, 699-710, 1996.


\bibitem{C-G-G-H}
Carlson, J.; Green,M.; Griffiths, P.; Harris, J.; {\it Infinitesimal variations of Hodge structures (I)}, Compos.
Math. 50, 109-205, 1983.



\bibitem{C-T}
Carlson, J.; Toledo, D.; {\it Discriminant complements and kernels of monodromy representations}, Duke Math. J. Vol.97, No.3, 621-648, 1999.

\bibitem{Cynk-van Straten}
Cynk, S.; van Straten, D.; {\it Infinitesimal deformations of double covers of smooth algebraic varieties}, Math. Nachr. 279(7), 716-726, 1996.

\bibitem{D-HodgeII}
Deligne,P; {\it Th\'{e}orie de Hodge II}, Publ. Math. I.H.E.S. 40, 5-57,1971.

\bibitem{D-Weil}
Deligne,P; {\it La conjecture de Weil I, II}, Publ. Math. I.H.E.S. 43, 273-307,1974; 52, 137-252, 1980.


\bibitem{D}
Deligne,P; {\it Un th\'{e}or\`{e}me de finitude pour la monodromie}, Discrete Groups in Geometry and Analysis,
Birkhauser, 1-19, 1987.

\bibitem{DM}
Deligne,P;Mostow,G.D; {\it Monodromy of hypergeometric functions
and non-lattice integral monodromy}, Publ. Math. I.H.E.S., Tome
63, 5-89, 1986.


\bibitem{DKV}
Dolgachev,I.; van Geemen,B.; Kond\=o,S; {\it A complex ball
uniformization of the moduli space of cubic surfaces via periods
of K3 surfaces}.  J. Reine Angew. Math.  588  (2005), 99-148.

\bibitem{DK}
Dolgachev,I.; Kond\=o,S; {\it Moduli of K3 surfaces and complex ball quotients}.  Arithmetic and Geometry Around Hypergeometric Functions, Lecture Notes of a CIMPA Summer School held at Galatasaray University, Istanbul, 2005.



\bibitem{DO}
Dolgachev,I; Ortland,D; {\it Point sets in projective spaces and
theta functions}, Ast\'{e}risque 165, 1988.


\bibitem{FL}
Friedmann,R.; Laza,R.; {\it Semi-algebraic horizontal subvarieties of Calabi-Yau type}, Duke Math J, 162(2013), no. 12, 2077-2148.


\bibitem{EV}
Esnault,H; Viehweg,E; {\it Lectures on Vanishing
Theorems}, DMV Seminar, 20, Birkh. Verlag, Basel etc.
(1992).

\bibitem{FH}
Fulton, W. Harris, J.: {\it Representation theory, A first course}, GTM 129.



\bibitem{GSSZ}
Gerkmann, R; Sheng, M; Van Straten, D; Zuo, K; {\it On the Monodromy of the Moduli Space of Calabi-Yau Threefolds Coming From Eight Planes in $\P^{3}$}, Mathematische Annalen,  Vol. 355,  187-214, 2013.

\bibitem{Gri}
Griffiths, P; {\it Infinitesimal variations of Hodge structure. III. Determinantal varieties and the infinitesimal invariant of normal function},
Compositio Math. 50 (1983), no. 2-3, 267-324.

\bibitem{Griffiths}
Griffiths, P; {\it Topics in transcendental algebraic geometry}. Edited by Phillip Griffiths. Annals of Mathematics Studies, 106. Princeton University Press, Princeton, NJ, 1984.

\bibitem{G}
Gross,B; {\it A remark on tube domains}, Math. Res. Lett., Vol.1, 1-9,
1994.

\bibitem{Jost-Zuo}
Jost, J.; Zuo, K.;{\it Harmonicmaps and Sl(r,C)-representations of fundamental groups of quasiprojective
manifolds}, J. Algebraic Geom., 5(1), 77-106, 1996.

\bibitem{K-R}
Kudla, S; Rapoport, M; {\it On occult period maps}, Pacific J. Math., 260, 565- 581, 2012.




\bibitem{Mochizuki}
Mochizuki, T.;{\it Asymptotic behaviour of tame nilpotent harmonic bundles with trivial parabolic
structure},  J. Differ. Geom., 62(3), 351-559, 2002.

\bibitem{La}
Laza, R.; {\it The moduli space of cubic fourfolds via the period map},
Ann. of Math. 172 (2010), no. 1, 673-711.


\bibitem{L}
Looijenga, E.; {\it Uniformization by Lauricella functions an overview of the theory of Deligne–Mostow,
arithmetic and geometry around hypergeometric functions},  Prog. Math.,  260, 207-244, 2007.

\bibitem{M}
Moonen, B;{\it Special subvarieties arising from families of cyclic covers of the projective
line}, Doc. Math., Vol.15, 793-819, 2010.


\bibitem{MSY}
Matsumoto,K; Sasaki,T.; Yoshida,M.;{\it The monodromy of the period map
of a 4-parameter family of K3 surfaces and the hypergeometric
function of type (3,6)}, International Journal of Mathematics, Vol 3,
No.1, 1-164, 1992.

\bibitem{Mostow0}
Mostow,G;{\it Generalized Picard lattices arising from half-integral
conditions}, Inst. Hautes \'{E}tudes Sci. Publ. Math. No. 63 (1986),
91-106.


\bibitem{Mostow}
Mostow,G;{\it On discontinuous action of monodromy groups on the
complex $n$-ball}, J. Amer. Math. Soc. 1 (1988), no. 3, 555-586.


\bibitem{Pa}
Paranjape,K.;{\it Abelian varieties associated to certain K3 surfaces},
Compositio Math. 68 (1988), no. 1, 11-22.

\bibitem{PS}
Peters,C; Steenbrink,J; {\it Mixed Hodge Structures}, Ergebnisse der Mathematik und
ihrer Grenzgebiete. 3. Folge 52, Springer-Verlag, Berlin, 2008.

\bibitem{Rhode}
Rohde, J.: {\it Cyclic coverings, Calabi-Yau manifolds and complex multiplication}, Lecture Notes in
Mathematics. Springer, Berlin (1975). 2009



\bibitem{SXZ}
Sheng, M.; Xu, J; Zuo, K.; {\it Maximal families of Calabi-Yau manifolds with minimal length Yukawa coupling}, Communications in Mathematics and Statistics, Vol.1, No.1, 73-92, 2013.




\bibitem{SZ}
Sheng, M.; Zuo, K.; {\it Polarized variation of Hodge structures of Calabi-Yau type and
characteristic subvarieties over bounded symmetric domains}, Mathematische Annalen, vol. 348, 213-225, 2010.

\bibitem{Simpson90}
Simpson, C.; {\it Harmonic bundles on noncompact curves}, J. Amer. Math. Soc. 3(3), 713-770, 1990.

\bibitem{Simpson92}
Simpson, C.; {\it Higgs bundles and local systems}, Publ. Math. I.H.E.S. 75, 5-95, 1992.


\bibitem{Simpson94}
Simpson, C.; {\it Moduli of representations of the fundamental group of a smooth projective varieites II}, Publ. Math. I.H.E.S. 80, 5-79, 1994.



\bibitem{Te}
Terasoma, T; {\it Complete intersections of hypersurfaces-the Fermat case and the quadric case}, Japan J.
Math. 14(2), 309-384, 1988.




\bibitem{T}
Terasoma, T; {\it Infinitesimal variation of Hodge structures and the weak global Torelli theorem for complete intersections}, Annals of Mathematics, Vol. 132, No. 2, 213-225, 1990.

\bibitem{T-fundamental groups}
Terasoma, T; {\it Fundamental groups of moduli spaces of hyperplane configurations}, preprint, 1994.

\bibitem{VZ}
Viehweg, E., Zuo, K.: {\it A characterization of certain Shimura curves in the moduli stack of abelian
varieties},  J. Differ. Geom. 66(2), 233-287, 2004.

\end{thebibliography}
\end{document}